\documentclass[leqno]{amsart}%
\usepackage{hyperref}
\usepackage{pdfsync}
\usepackage{dsfont}
\usepackage{color}
\usepackage{amsthm}
\usepackage{braket}
\usepackage{bm}
\usepackage{amsmath}
\usepackage{graphicx}
\usepackage{caption}
\usepackage{subcaption}
\usepackage{pb-diagram}
\usepackage{amsfonts}
\usepackage{amssymb}%
\providecommand{\U}[1]{\protect\rule{.1in}{.1in}}
\newtheorem{theorem}{Theorem}[section]

\newtheorem{definition}[theorem]{Definition}

\newtheorem{lemma}[theorem]{Lemma}

\newtheorem{proposition}[theorem]{Proposition}

\theoremstyle{remark}
\newtheorem{remark}[theorem]{Remark}
\newcommand{\R}{\mathbb{R}}
\newcommand{\N}{\mathbb{N}}

\newcommand{\te}{\textrm}

\newcommand{\veps}{\varepsilon}
\DeclareMathOperator{\inte}{int}

\DeclareMathOperator{\dist}{dist}
\DeclareMathOperator{\relint}{resint}
\DeclareMathOperator{\diam}{diam}

\DeclareMathOperator{\esssup}{esssup}
\DeclareMathOperator{\Lip}{Lip}

\newcommand{\charf}{\mathds{1}}

\definecolor{mygreen}{rgb}{0.1,0.75,0.2}

\title{On the rate of convergence of empirical measures in $\infty$-transportation distance}
\author{Nicol\'as Garc\'ia Trillos and  Dejan Slep\v{c}ev}


\begin{document}
\newcounter{broj1}
\date{\today}
\maketitle

\begin{abstract}
We consider random i.i.d. samples of absolutely continuous measures on bounded connected domains.
We prove an upper bound on the $\infty$-transportation distance between the measure and the empirical measure of the sample. The bound is optimal in terms of scaling with the number of sample points. 
\end{abstract}

\section{Introduction}
Consider a bounded, open set $D \subset \R^d$.
Given two probability measures $\nu$ and $\mu$ on $D$ the \emph{$\infty$-transportation distance }between $\nu$ and $\mu$ is defined by
\begin{equation} \label{WINF}
d_\infty(\nu, \mu):= \inf \left\{  \esssup_\gamma \{ |x-y| \::\: (x,y) \in D \times D \}  \: : \: \gamma \in \Gamma(\nu, \mu)  \right\},
\end{equation}
where $ \Gamma(\nu, \mu)$ is the set of all {\em couplings}  (transportation plans) between $\nu$ and $\mu$, that is, the set of all  probability measures on $D \times D$ for which the marginal on the first variable is $\nu$ and the marginal on the second variable is $\mu$. More precisely:
\[  \Gamma(\nu, \mu) = \{ \gamma \in \mathcal P(D \times D)\::\: (\forall A - \te{Borel}) \;\, 
\gamma(A \times D) = \nu(A), \; \gamma(D \times A) = \mu(A) \}.  \]

We consider the $\infty$-transportation distance between a given measure, $\nu$, and the empirical measure associated to a random i.i.d sample drawn from the measure $\nu$. We consider $\nu$ which is absolutely continuous with respect to the Lebesgue measure. Our main result is the following upper bound.
\begin{theorem}
Let $D\subseteq \R^d$ be a bounded, connected, open set with Lipschitz boundary. Let $\nu$ be a probability measure on $D$ with density $\rho: D \rightarrow (0,\infty)$ such that there exists $\lambda \geq 1$ for which 
\begin{equation}
(\forall x \in D) \quad \frac{1}{\lambda} \leq \rho(x) \leq \lambda.
  \label{DensityBound}
  \end{equation}
 Let $X_1, \dots, X_n, \dots$ be i.i.d. samples from $\nu$. Consider $\nu_n$ the empirical measure 
$$  \nu_n:=  \frac{1}{n}\sum_{i=1}^{n} \delta_{X_i}.  $$
Then, for any fixed $\alpha >2$, except on a set with probability $O(n^{-\alpha/2})$, 
$$ d_\infty(\nu , \nu_n) \leq C
\begin{cases}
 \frac{\ln(n)^{3/4}}{n^{1/2}},  & \te{if } d =2, \medskip \\
 \frac{\ln(n)^{1/d}}{n^{1/d}},  & \te{if } d \geq 3,
\end{cases}
$$
where $C$ depends only on $\alpha $, $D$, and $\lambda$.
\label{main}
\end{theorem}
We also establish that the bound above is optimal in terms of scaling in $n$.

\subsection{Background.}
The $\infty$-transportation distance $d_\infty(\nu, \mu)$ is the least possible maximal distance that a transportation plan between $\nu$ and $\mu$ has to move the mass by. 
Related to it is the $p$-transportation distance (i.e. Monge-Kantorovich-Rubinstein or $p$-Wasserstein distance)  which measures the average of the power of the distance the mass is moved by: For $1 \leq p < \infty$
\[ d_p(\nu, \mu):=  \left ( \inf  \left\{  \int |x-y|^p d \gamma(x,y) \; :  \: \gamma \in \Gamma(\nu, \mu)  \right\} \right )^\frac{1}{p}. \]
It metrizes  the weak convergence of measures on $D$ (for $D$  bounded).
If follows from the work of Dudley \cite {Dud} that for $1 \leq p <\infty$ and  $d \geq 3$ and under rather general conditions on $\nu$ (weaker than ones assumed in Theorem \ref{main}) that the expected $p$-transportation distance between a measure $\nu$
and the empirical measure $\nu_n$ scales as $n^{-1/d}$:
\[ d_p(\nu, \nu_n) \sim  n^{-1/d}  \quad \te{ for } d \geq 3. \]

A related problem consists on comparing two measures $\nu_n$ and $\mu_n$ both of which are discrete measures with the same number of points of the same mass. Then the $\infty$-transportation distance $d_\infty(\nu_n, \mu_n)$ is also known as the min-max matching distance. There are a number of works on the matchings in the case that $\mu_n$ and $\nu_n$ are measures on cube the $(0,1)^d$ and $\nu_n$ is the empirical measure of a i.i.d. samples drawn from the Lebesgue measure and $\mu_n$ is either another empirical measure of another independent sample or a measure supported on a regular grid. It is worth remarking that the discrete matching results imply the estimates on the distance between $\nu_n$ and $\nu$. The converse also holds. 
 
Ajtai, Koml{\'o}s, and Tusn{\'a}dy in \cite{AKT} showed optimal bounds on the $p$-transportation distance, for $1 \leq p < \infty$, between two empirical measures sampled from the Lebesgue measure on a square. That is they showed that if $X_1, \dots, X_n, \dots \in (0,1)^2$ and 
$Y_1, \dots, Y_n, \dots \in (0,1)^2$ are two independent samples and $\mu_n = \frac{1}{n} \sum_{i=1}^n \delta_{Y_i}$, while $\nu_n$ is as before then the minimum
over all permutations $\pi$ of $\{1, \dots, n\}$ satisfies
\[ \min_\pi \frac{1}{n} \sum_{i=1}^n |X_{\pi(i)} - Y_{i}| \leq C \sqrt{\frac{\log n}{n}} \]
with probability $1- o(1)$. They introduced the technique of obtaining probabilistic estimates by dyadically dividing the cube into $2^k$ subcubes, obtaining a matching estimate at the fine level and estimating the transformations needed to bridge different scales to obtain an upper bound on the total distance. Our proof also relies on a similar decomposition of the domain. 

Dobri\'c and Yukich \cite{DobYuk}, Talagrand \cite{Talagrand}  and Talagrand, Yukich,\cite{TalagrandYukich}, Bolley,  Guillin, and Villani \cite{BolGuiVil07}, Boissard \cite{Boissard},
and others later refined these results and obtained more precise information on the distribution of the $p$-transportation distance between a measure on a cube and the empirical measure. 

%
%
%

For the $\infty$-transportation distance obtaining estimates is more delicate, since almost all of the mass needs to be matched within the desired distance to obtain the bound. Furthermore
 the optimal scaling itself has a logarithmic correction compared to the case $p< \infty$. 

The optimal scaling in dimension $d=2$, for $\nu$ being the Lebesgue measure, was obtained by Leighton and Shor \cite{LeightonShor}. They consider i.i.d random samples $X_1, \dots, X_n$ 
distributed according to the Lebesgue measure and points $Y_1, \dots, Y_n$ on a regular grid. They 
showed that  there exist $c>0$ and $C>0$ such that with very high probability:
\begin{equation}
 \frac{c(\log n)^{3/4}}{n^{1/2}} \leq \min_{\pi}  \max_{i} |X_{\pi(i)} - Y_{i}|  \leq \frac{ C(\log n)^{3/4}}{n^{1/2}},
 \label{LeighShor1}
\end{equation}
where $\pi$ ranges over all permutations of $\left\{1, \dots , n \right\}$. In other words,  when $d=2$, with high probability the $\infty$-transportation distance between the measure $\mu_n$ and the measure $\nu_n$ is of order $\frac{ (\log n)^{3/4}}{n^{1/2}}$. 

For $d\geq 3$,  Shor and Yukich \cite{ShorYukich} proved the analogous result on $(0,1)^d$ with $\nu$ being the Lebesgue measure restricted to $(0,1)^d$.
They showed that there exist $c>0$ and $C>0$ such that with very high probability:
\begin{equation}
 \frac{c(\log n)^{1/d}}{n^{1/d}} \leq \min_{\pi}  \max_{i} |X_{\pi(i)} - Y_{i}|  \leq \frac{ C(\log n)^{1/d}}{n^{1/d}}.
 \label{ShorYukich1}
\end{equation}
The result in dimension $d \geq 3$ is based on the matching algorithm introduced by Ajtai, Koml\'os, and Tusn\'ady in \cite{AKT}.  
For $d=2$ the AKT scheme still gives an upper bound, but not a sharp one.
As remarked in \cite{ShorYukich}, there is a crossover in the nature of the matching when $d=2$: for $d \geq 3$, the matching length between the random points and the points in the grid is determined by the behavior of the points locally, for $d=1$ on the other hand, the matching length is determined by the behavior of  random points globally, and finally for $d=2$ the matching length is determined by the behavior of the random points at all scales. At the level of the AKT scheme this means that for $d \geq 3$ the major source of the transportation distance is at the finest scale, for $d=1$ at the coarsest scale, while for $d=2$ distances at all scales are of the same size (in terms of how they scale with $n$). The sharp result in dimension $d=2$ by Leighton and Shor required a more sophisticated matching procedure; an alternative proof was  provided by Talagrand \cite{Talagrand} who also provided 
more streamlined and conceptually clear proofs in \cite{TalagrandGenericChain,TalagrandNewBook}.


In this paper, our main contribution is that we extend the previous results to general domains and general densities. 

\subsection{Outline of the approach.}

One of the main steps in the proof of Theorem  \ref{main} consists on establishing estimates on the $\infty$-transportation distance between two measures which are absolutely continuous  with respect to the Lebesgue measure and whose densities are bounded from above and below by positive constants.  We prove the following result which is of interest on its own.
\begin{theorem}
Let $D \subset \R^d$ be a bounded, connected, open set with Lipschitz boundary. Let $\nu_1, \nu_2$ be measures on $D$ of the same total mass:  $  \nu_1(D)= \nu_2(D)$. Assume the measures are absolutely continuous with respect to Lebesgue measure and let $\rho_1$ and $\rho_2$ be their densities. Furthermore assume that for some $\lambda > 1$, for all $x \in D$
\begin{equation}
\frac{1}{\lambda} \leq \rho_i(x) \leq \lambda  \quad \te{ for } i=1,2.
\label{DensityBoundLInfinity}
\end{equation}
Then, there exists a constant $C(\lambda, D)$ depending only on $\lambda$ and $D$ such that for all
 $\nu_1$, $\nu_2$ as above
$$  d_\infty(\nu_1, \nu_2) \leq C(\lambda, D) \| \rho_1 - \rho_2\|_{L^\infty(D)}. $$
\label{PropositionLinfinityWinfinityGeneralDomain}
\end{theorem}

We use this result in proving Theorem \ref{main}. We first consider the domain $D=(0,1)^d$. Note that when the density $\rho$ is constant the result was obtained by Shor and Yukich \cite{ShorYukich} in case $d\geq 3$ and by Leighton and Shor \cite{LeightonShor} in case $d=2$. 
In dimensions $d\geq 3$ we use a dyadic decomposition similar to the one introduced by Ajtai, Koml{\'o}s, and Tusn{\'a}dy (also used by Shor and Yukich). However the fact that we adjust the densities and not the geometry of the subdomains makes it easier to handle general densities. 
 We remark that the probabilistic estimates in \cite{ShorYukich} are similar to the ones we use in our proof of Theorem \ref{main} for $d \geq 3$.
To obtain the optimal scaling when $d=2$ a more subtle approach is needed. Talagrand's  \cite{TalagrandGenericChain} proof of Leighton and Shor's theorem provides flexible tools which we adapt to nonuniform densities. 

We note that having the result on $(0,1)^d$ implies the result on any domain that is bi-Lipshitz homeomorphic to $(0,1)^d$.

To  prove Theorem \ref{main} on general domains we partition them into finite number of subdomains which can be transformed via a bi-Lipschitz map to the unit cube. The difficulty which arises is that the empirical measure on the subdomain may not have the same total mass as the restriction of the measure to the subdomain. The mass discrepancy is small, but since we seek estimates in $\infty$-transportation distance, the mass discrepancy needs to be carefully redistributed. Thus we introduce special ways to partition domains which enable for the appropriate mass exchange between subdomains.
We call the domains \emph{well partitioned} (see Definition \ref{WP}) if they admit the desired partitioning. We prove the Theorem \ref{main} on well partitioned domains using induction on the number of subdomains needed in the partition.

We then show that all connected, bounded domains with smooth boundary  can be well partitioned by a careful geometric argument which uses Voronoi tessellation.  The final step in proving Theorem \ref{main} consists on reducing the problem to domains which are well partitioned by proving that it is possible to find a bi-Lipschitz homeomorphism between an arbitrary open, connected, bounded set $D$ with Lipschitz boundary and a domain which is well partitioned. In fact, by a result in \cite{BallZ} an open and bounded domain $D$ with Lipschitz boundary is bi-Lipschitz homeomorphic to an open domain with smooth boundary. Finally, Proposition \ref{HPropertySmoothDomains} says that open and bounded domains with smooth boundary are well partitioned.

%

The paper is organized as follows. In Subsection \ref{SubsecPrelim} we introduce notation and present some results related to the $\infty$-transportation distance. In Section \ref{SecMatchingBox} we start by proving Theorem \ref{PropositionLinfinityWinfinityGeneralDomain} for $D=(0,1)^d$ and then prove Theorem \ref{main} for $D=(0,1)^d$, in Subsection \ref{SecMatchingBoxdgeq3} when $d\geq 3$ and in Subsection \ref{SecMatchingBoxdeq2} when $d=2$. In Section \ref{SecMatchingD} we define well partitioned domains and prove Theorem \ref{PropositionLinfinityWinfinityGeneralDomain} in full generality, then we prove Theorem \ref{main}. Finally, in the Appendix we prove Proposition \ref{HPropertySmoothDomains} which states that open and bounded domains with smooth boundary are well partitioned.

\subsection{Preliminaries and notation.} \label{SubsecPrelim}Let $D$ be an open and bounded domain in $\R^d$. Given a finite Borel measure $\nu $ 
and a Borel map $T: D \rightarrow D$, the \emph{push-forward} of $\nu$, denoted by $T_\sharp \nu$, is the measure such that for all $A \in \mathfrak{B}(D)$
$$ T_\sharp \nu( A) := \nu (T^{-1}(A)).$$
We say that a Borel map $T: D \rightarrow D$ is a \emph{transportation map} between $\nu $ and $\mu$ if $T_\sharp \nu = \mu$. Note that a transportation map $T$ between $\nu$ and $\mu$ induces the coupling $\gamma_T $  given by:
$$  \gamma_T := (Id \times T)_\sharp \nu.    $$
A natural question that arises from the connection between transportation maps and transportation plans is the following: in the definition of $d_\infty(\nu, \mu)$ can we restrict our attention to couplings induced by transportation maps?. The answer to this question is affirmative in case the measure $\nu$ is absolutely continuous with respect to the Lebesgue measure. In fact, this is one of the results in \cite{Champion}, where it is proved that there exists solutions to the problem \eqref{WINF} 
which are also $\infty$-\emph{cyclically monotone}, that if $\nu \ll \mathcal{L}^d$, are induced by transportation maps. In this paper $\nu$ is taken to be  $d \nu = \rho dx$, where $\rho$ is bounded above and below by positive constants and so in this setting the results in \cite{Champion} can be stated as follows: if $\nu(D) = \mu (D)$, then there exists a transportation map $T^*: D \rightarrow D$ with $T^*_\sharp \nu = \mu$ and such that 
\begin{equation}
 d_\infty(\nu, \mu)= \| T^* - Id \|_{L^\infty(D)} . 
 \label{MapsVsPlans}
 \end{equation}
The question of uniqueness of the optimal transportation map $T^*$, although interesting on its own, is not of importance for the results we present in this paper. Nevertheless, it is worth mentioning that if  $\mu$ is concentrated on finitely many points then, the transportation map $T^*$ for which \eqref{MapsVsPlans} holds is unique; this is the content of Theorem 5.4 in \cite{Champion}. In particular, if $\mu$ is taken to be  $\nu_n$, where $\nu_n$ is the empirical measure associated to data points $X_1, \dots, X_n$ sampled from $\nu$, then the uniqueness of the optimal transportation map  is guaranteed.

We remark that for any transportation map $T_n$ between $\nu$ and $\nu_n$ it holds that
 $$  d_\infty(\nu, \nu_n)  \leq \|T_n-Id\|_{L^\infty(D)}.  $$
Thus, we can estimate $d_\infty(\nu, \nu_n)$ by estimating the right hand side of the previous expression for some $T_n$.

\section{The matching results for $(0,1)^d$.}
\label{SecMatchingBox}
The first goal of this section is to prove Theorem \ref{PropositionLinfinityWinfinityGeneralDomain} for $D=(0,1)^d$. In order to do this we need a few preliminary lemmas.
 
\begin{lemma}
Let $Q \subseteq \R^d$ be a rectangular box (rectangular parallelepiped). Let $Q_1, Q_2$ be the rectangular boxes obtained from $Q$ by bisecting one of its sides.  Let $ \rho: Q \rightarrow (0,\infty)$ be given by
\begin{equation*}
\rho(x) :=\left\{ \begin{array}{l}
  c_1 , \: \te{if } \: x \in Q_1 \\
  c_2 , \: \te{if } \: x\in Q_2,
\end{array} \right.
\end{equation*}
where $c_1,c_2>0$ are such that $ 1= \frac{c_1}{2}+\frac{c_2}{2} $. Denote by $\nu$ the measure with $d\nu= \rho(x)dx$ and let $\nu_0$ be the Lebesgue measure restricted to $Q$. Then, 
$$ d_\infty(\nu_0, \nu)  \leq \frac{L}{2}  \left| \frac{\nu(Q_1)} {\nu_0(Q_1)}-1 \right|, $$
where $L$ is the length of the side of $Q$ bisected to generate $Q_1$ and $Q_2$.
\label{LemmaRectangleUniformtoTwoValues}
\end{lemma}

\begin{proof}
Without the loss of generality we can assume that $Q= [0,L] \times \hat Q$ where $\hat Q$ is a $d-1$ dimensional rectangular box. Thus $Q_1=[0, \frac{L}{2}] \times  \hat Q$ and $Q_2=[ \frac{L}{2}, L] 
\times \hat Q$. Note that the condition $1=\frac{c_1}{2}+ \frac{c_2}{2}$ is equivalent to $\nu(Q)= \nu_0(Q)$.  Let us introduce auxiliary functions $h(t)= c_1 \charf_{[0,\frac{L}{2}] }(t) + c_2\charf_{(\frac{L}{2},L]}(t)$ and $f(t) = \charf_{[0,L]}(t)$. 
For $t \in [0,L]$ let $F(t) = \int_0^t  ds= t$ and $H(t) = \int_0^t h(s) ds$, that is, 
\begin{equation*}
H(t):=\left\{
\begin{array}{ll}
c_1 t& \te{ if } 0\leq t \leq \frac{L}{2} \\
\frac{c_1}{2} L + c_2 (t-\frac{L}{2}) & \te{ if } \frac{L}{2} \leq t \leq L.
\end{array}\right.
\end{equation*}

A direct computation shows that 
\begin{equation*}
H^{-1}\circ F(t):=\left\{
\begin{array}{ll}
\frac{t}{c_1} & \te{ if } 0 \leq t \leq \frac{c_1L}{2} 
\\
\frac{t}{c_2} + \frac{L}{2}( 1 - \frac{c_1}{c_2})& \te{ if } \frac{c_1L}{2} \leq t \leq L.
\end{array}\right.
\end{equation*}
Notice that the map $T_1:= H^{-1}\circ F$ is a transportation plan between the measures $ dt$ and $h(t) dt$.
Therefore, $T= T_1 \times I_{d-1}$ is a transportation plan between $\nu_0$ and $\nu$.

A direct computation shows that 
$$  | T(x) - x    | = | H^{-1}\circ F (x_1) - x_1|  \leq \frac{L}{2} | c_1-1|, $$
for all $x \in Q$. Since $c_1= \frac{\nu (Q_1) }{\nu_0(Q_1)} $, we conclude from the previous inequality that:
$$  \| T - Id \|_{L^\infty(Q)} \leq \frac{L}{2} \left|  \frac{\nu (Q_1) }{\nu_0(Q_1)} -1 \right|,$$
which implies the result.
\end{proof}

 \begin{lemma}
 Let $\rho: (0,1)^d \rightarrow (0, \infty)$ be integrable and let $\nu$ be the measure given by $d \nu = \rho  dx$. Let   $a = \int_{(0,1)^d}\rho(x)dx $ and denote by  $\nu_0$ the measure on $(0,1)^d$ given by $d \nu_0 = a dx$. Then,          
$$ d_\infty(\nu_0, \nu) \leq \frac{\overline C(d)}{a} \|a- \rho\|_{L^\infty((0,1)^d)},  $$
where $\overline C(d)$ is a constant that depends on $d$ only.
 \label{LemmaLInfinityBoundUniformVsNoise}
 \end{lemma}
\begin{proof}
Given that
$$ d_\infty(\nu_0, \nu) = d_\infty \left(\frac{1}{a} \nu_0 , \frac{1}{a}  \nu \right),  $$
by rescaling the densities, it is enough to prove the result for $a=1$. 

Consider first the case that $\|1- \rho\|_{L^\infty((0,1)^d)} < 1/2$.  

\emph{Step 1.} 
For every $k \in \N$ we consider a partition of $[0,1]^d$ into a family $\mathcal{G}_k$ of $2^k$ rectangular boxes. The boxes are constructed recursively. Let 
$\mathcal{G}_{0} = \{  (0,1)^d \}$. Given the collection of boxes $\mathcal{G}_k$, the collection of rectangular boxes $\mathcal{G}_{k+1}$ is obtained by bisecting each of the rectangular boxes  belonging to $\mathcal{G}_k$ through their longest side. We note that all boxes in $\mathcal G_k$ have volume $\frac{1}{2^k}$ and have the same diameter (which depends only on $k$ and $d$).

Consider  $\rho_0:=1$ and for all $k>0$ and all $Q \in \mathcal G_k$ let:
\begin{equation} \label{rhok}
\rho_k(x):=  \frac{1}{\nu_0(Q)} \int_Q \rho(z) dz =\frac{\nu(Q)}{\nu_0(Q)}  \qquad \te{for all } x \in Q.
\end{equation}
Let $\nu_k$ be the measure with density $\rho_k$. The assumption $\|1 - \rho \|_{L^\infty((0,1)^d)} < \frac12$ implies $\frac12 \leq \rho \leq \frac32$ and consequently for all $k$, $\frac12 \leq \rho_k \leq \frac32$.

Note that for all $Q \in \mathcal G_k$ and all $j \geq k$, $\nu_j(Q) = \nu_k(Q)= \nu(Q)$.  
We denote by $\nu_k \llcorner_Q$, the restriction of the measure $\nu_k$ to $Q$.
The relation of $\nu$ to $\nu_k$ on $Q$ is similar to the one of $\nu$ to $\nu_0$ on $(0,1)^d$, only that the scale is smaller. We show that estimates on $\infty$-transportation distance on the finer scale lead to the desired estimates on the macroscopic scale. Note that
\begin{equation}
d_\infty(\nu_k, \nu_{k+1})  \leq \max_{Q \in \mathcal{G}_k} d_\infty(\nu_k \llcorner_Q , \nu_{k+1}\llcorner_Q ),
\label{MatchinLinifinitySubcubes0}
\end{equation}
and that
\begin{equation}
d_\infty(\nu_k, \nu)  \leq \max_{Q \in \mathcal{G}_k} d_\infty(\nu_k \llcorner_Q , \nu\llcorner_Q ) \leq \max_{Q \in \mathcal{G}_k} \diam(Q)  \leq \frac{C}{2^{k/d}} ,
\label{MatchinLinifinitySubcubesEmpirical0}
\end{equation}
where $C$ is a constant only depending on $d$. 

\emph{Step 2.} 
Let $Q \in \mathcal{G}_k$ and let $Q_1, Q_2 \in \mathcal G_{k+1}$ be the two sub-boxes of $Q$. Then, $\nu_k(Q_1) = \frac{1}{2}\nu_k(Q) $ and $\nu_0(Q_1) = \frac12 \nu_0(Q).$ It follows that
\begin{align*}
| \nu(Q_1) - \nu_k(Q_1)| & \leq   | \nu(Q_1) - \nu_0(Q_1)| + | \nu_0(Q_1) - \nu_k(Q_1)  |
\\ &  = \| \rho - 1\|_{L^\infty((0,1)^d)}\nu_0(Q_1) + \left| \frac{1}{2}\nu _0(Q) - \frac{1}{2}\nu_k(Q) \right|
\\& \leq \| \rho - 1\|_{L^\infty((0,1)^d)}\nu_0(Q_1) +  \frac{1}{2}\| \rho-1 \|_{L^\infty((0,1)^d)}\nu_0(Q)
\\&= 2 \|\rho-1\|_{L^\infty((0,1)^d)} \nu_0(Q_1).
\end{align*}
Therefore,
\begin{equation}
\frac{| \nu(Q_1)  - \nu_k(Q_1)  |}{\nu_k(Q_1)}  \leq \frac{2 \|\rho- 1\|_{L^\infty((0,1)^d)} \nu_0(Q_1) }{ \nu_0(Q_1) /2} = 4\|\rho-1\|_{L^\infty((0,1)^d)}.  
\label{EqnDiscrepancy}
\end{equation}

\emph{Step 3.} 
For a fixed cube $Q \in \mathcal{G}_k$, denote the value of $\rho_k$ in $Q$ by $b$. Then,
$$ d_\infty(\nu_k \llcorner_Q , \nu_{k+1}\llcorner_Q ) = d_\infty \left(\frac{1}{b}\nu_k \llcorner_Q , \frac{1}{b}\nu_{k+1}\llcorner_Q \right).    $$
By Lemma \ref{LemmaRectangleUniformtoTwoValues} and by \eqref{EqnDiscrepancy}  we have
\begin{align*}    
 d_\infty\left(\frac{1}{b}\nu_k \llcorner_Q , \frac{1}{b}\nu_{k+1}\llcorner_Q \right) & \leq \frac{1}{2^{k/d}}   \left| \frac{\nu(Q_1)}{\nu_k(Q_1)} - 1 \right| 
 \\ & \leq  \frac{4}{2^{k/d}}\|\rho-1\|_{L^\infty((0,1)^d)} 
\end{align*}
From \eqref{MatchinLinifinitySubcubes0} and the previous inequality it follows that for every $k \in \N$
\begin{equation*}
d_\infty\left( \nu_k, \nu_{k+1} \right) \leq \frac{4}{2^{k/d}}\|\rho-1\|_{L^\infty((0,1)^d)}.
\end{equation*}
 Choose $\tilde{k}$ such that $2^{-\tilde{k}/d } \leq \| \rho- 1\|_{L^\infty}$. From the previous inequality and \eqref{MatchinLinifinitySubcubesEmpirical0} we deduce that
 \begin{align}
 \begin{split}
  d_\infty(\nu_0, \nu)&  \leq   \sum_{k=0}^{\tilde{k}-1} d_\infty(\nu_k , \nu_{k+1}) + d_\infty(\nu_{\tilde{k}}, \nu )   
  \\ & \leq  4 \| \rho-1\|_{L^\infty((0,1)^d)} \sum_{k=0}^{\tilde{k}-1} \frac{1}{2^{k/d}} + C \frac{1}{2^{\tilde{k}/d  }}
  \\& \leq C(d) \| \rho-1\|_{L^\infty((0,1)^d)},
  \end{split} \label{temp1}
 \end{align}
 which shows the desired result.
 \medskip
 
We now turn to case  $\|\rho-1\|_{L^\infty((0,1)^d)} \geq 1/2$. The desired estimate follows from
$$  d_\infty(\nu_0, \nu) \leq \diam((0,1)^d) = \sqrt{d} \leq 2 \sqrt{d} \| 1- \rho\|_{L^\infty((0,1)^d)}.  $$

In conclusion taking the larger of the constants of the cases above, $\overline C = \max\{C(d), 2 \sqrt{d} \}$, provides the desired estimate.
\end{proof}


\begin{proof}[Proof of Theorem \ref{PropositionLinfinityWinfinityGeneralDomain} for $D=(0,1)^d$.]
Suppose first that $\|\rho_1 - \rho_2\|_{L^\infty((0,1)^d)} \leq \frac{1}{2\lambda}$. 
Let $g(x) =  \rho_1(x)-\rho_2(x) + \frac{1}{\lambda}$. Note that $g \geq 0$ and that
\begin{align*}
\rho_1 &  = \left( \rho_2 - \frac{1}{\lambda}  \right)+ g\\
\rho_2 & = \left( \rho_2 - \frac{1}{\lambda} \right)+ \frac{1}{\lambda}.
\end{align*} 
By Lemma \ref{LemmaLInfinityBoundUniformVsNoise} and by \eqref{MapsVsPlans}, there exists a transportation map $T$ between the measures $gdx$ and  $\frac{1}{\lambda}dx$ such that 
\begin{equation*}
 \| T - Id \|_{L^\infty((0,1)^d)} \leq  \lambda C(d) \left\| g - \frac{1}{\lambda} \right \| _{L^\infty((0,1)^d)} =  \lambda C(d) \| \rho_1 - \rho_2\|_{L^\infty((0,1)^d)}.  
\end{equation*}
Note that 
\[ \gamma := (Id \times Id)_\sharp \left(\rho_2 - \frac{1}{\lambda}\right)dx + (Id \times T)_\sharp gdx\, \in \Gamma(\nu_1, \nu_2). \]
Moreover for $\gamma$-a.e. $(x,y) \in (0,1)^d \times (0,1)^d$,
$$ |x-y| \leq  \lambda C( d) \|\rho_1- \rho_2\|_{L^\infty((0,1)^d)}. $$
Thus,
$$  d_\infty(\nu_1, \nu_2) \leq  \lambda C( d) \|\rho_1- \rho_2\|_{L^\infty((0,1)^d)}.  $$
To get our estimate in case $\|\rho_1 - \rho_2\|_{L^\infty} > \frac{1}{2\lambda}$ note that:
$$  d_\infty(\nu_1, \nu_2) \leq \diam((0,1)^d) = \sqrt{d} \leq 2\lambda \sqrt{d} \| \rho_1- \rho_2\|_{L^\infty((0,1)^d)}.  $$
 \end{proof}

\begin{remark}
Note that from the previous proof, Theorem \ref{PropositionLinfinityWinfinityGeneralDomain} is true for any domain $D$ of the form $D=(a_1,b_1) \times \dots \times (a_d, b_d)$. To deduce this fact, it is enough to consider a translation and rescaling of the coordinate axes to transform the rectangular box $D$ into the unit box $(0,1)^d$ and then use Theorem \ref{PropositionLinfinityWinfinityGeneralDomain} for the unit cube.
\label{LIninfinityEstimateRectangle}
\end{remark}

\subsection{The matching results for $(0,1)^d$: $d \geq 3$.}
\label{SecMatchingBoxdgeq3}
Now we prove Theorem \ref{main} for $D=(0,1)^d$ when $d\geq 3$. To achieve this it is useful to consider a partition of the cube $(0,1)^d$ into rectangular boxes analogous to the ones used in the proof of Lemma \ref{LemmaLInfinityBoundUniformVsNoise}. The main difference is that we divide rectangular boxes into sub-boxes of the same $\nu$-measure,  instead of the same Lebesgue measure.

Let $\rho:(0,1)^d \rightarrow (0, \infty) $ be a density function satisfying $1/\lambda \leq \rho \leq \lambda$. For every $k \in \N$ we construct a family $\mathcal{F}_k$ of $2^k$ rectangular boxes 
 which partition the cube $(0,1)^d$ with each rectangular box having $\nu$-volume equal to $\frac{1}{2^k}$ and aspect ratio (ratio between its longest side and its shortest side) controlled in terms of $\lambda$. 
We let $\mathcal{F}_{0} = \{(0,1)^d \}$.  For $k=1$ we construct rectangular boxes $Q_1$ and $Q_2$ by bisecting one of the sides (say the one lying on the first coordinate) of the cube $(0,1)^d$ using the measure $\nu$. That is, we define $Q_1:= (0,a) \times (0,1)^{d-1}$ and  $Q_2:=[a,1) \times (0,1)^{d-1}$ where $a\in (0,1)$ is such that $\nu_{Q_1}=1/2 \nu(Q)$ . Recursively, the collection of rectangular boxes at level $k+1$ is obtained by bisecting, according the measure $\nu$, each rectangular box from level $k$ through its \textbf{longest} side.

\begin{lemma}
The aspect ratio of every rectangular box in $\mathcal{F}_k$ is bounded by $2 \lambda^2$.
\label{AspectRatioLemma}
\end{lemma}
\begin{proof}
We show that for every $k\in \N$, every rectangular box in $\mathcal{F}_k$ has aspect ratio less than $2 \lambda^2$. The proof is by induction on $k$.

\emph{Base Case: }At level $k=1$ we consider $Q_1 = (0,a) \times (0,1)^{d-1}$, $a$ chosen so that $\nu(Q_1)=1/2$. Note that the aspect ratio of $Q_1$ is equal to $1/a$. Notice that, 

$$ \frac{1}{2} = \int_{Q_1} \rho(x) dx \leq a\lambda.   $$  
From this we conclude that the aspect ratio of $Q_1$ is no larger than $2\lambda$ and in particular no larger than $2\lambda^2$. By symmetry, the aspect ratio of $Q_2$ is no larger than $2 \lambda^2$.   
 
\emph{Inductive Step.} Suppose that the aspect ratio of every rectangular box in $\mathcal{F}_k$ is bounded by $2\lambda^2$. Let $Q$ be a rectangular box in $\mathcal{F}_{k+1}$. Note that $Q$ is obtained by bisecting (using the measure $\nu$) the longest side of a rectangular box $Q' \in \mathcal{F}_k$. Without the loss of generality we can assume that $\overline{Q'}= [a_1,b_1]\times [a_2, b_2] \times \dots \times [a_d, b_d]$ and that $\overline Q =  [a_1,c]\times [a_2, b_2] \times \dots \times [a_d, b_d]$ , where $a_1< c < b_1$. If $(a_1,c)$ is not the smallest side of $Q$ then the aspect ratio of $Q$ is no greater than the aspect ratio of $Q'$ and hence by the induction hypothesis is less than $2 \lambda^2$. If on the other hand $(a_1,c)$ is the smallest side of $Q$ then we let $(a_i,b_i)$ be the longest side of $Q$; the aspect ratio of $Q$ is then equal to $\frac{b_i-a_i}{c-a_1}$. Since $(a_1,b_1) $ is the longest side of $\tilde{Q}$, we have:
$$ \frac{b_i-a_i}{c-a_1} = \frac{b_1-a_1}{c-a_1}\frac{b_i-a_i}{b_1-a_1} \leq \frac{ b_1-a_1}{c-a_1}.$$ 
Finally, since
$$ \nu(Q) = \frac{1}{2} \nu(\tilde{Q}),    $$
we deduce that 
$$ (c-a_1) \lambda \geq \frac{1}{2\lambda} (b_1-a_1) .   $$
This implies the desired result.
\end{proof}

The proof of Theorem \ref{main} requires estimating how many of the sampled points fall in certain rectangles. These estimates rely on two concentration inequalities for binomial random variables, which we now recall. Let $S_m\sim\te{Bin}(m,p) $ be a binomial random variable, with $m$ trials and probability of success for each trial of $p$. Chernoff's inequality \cite{chern52} states that
\begin{equation} \label{chernoff}
P\left( \left| \frac{S_m}{m} - p \right| \geq t \right) \leq 2 \exp(-2 m t^2). 
\end{equation}
Bernstein's inequality \cite{bern24}, which is sharper for small values of $p$ gives that
\begin{equation} \label{bernstein}
P\left( \left| \frac{S_m}{m} - p \right| \geq   t \right) \leq 2 \exp \left(- \frac{ \frac12 m t^2}{ p(1-p) + \frac{1}{3} t } \right). 
\end{equation}

\begin{proof}[Proof of Theorem \ref{main} for $D=(0,1)^d$ when $d\geq 3$. ] 

\emph{Step 1.} Let $\rho_0:=\rho$ and let $\mu_0:= \nu$. For every $Q \in \mathcal F_k$, consider
\begin{equation} \label{T5rhok}
\rho_k(x) := \frac{\nu_n(Q)}{\nu(Q)} \rho(x) = \frac{\nu_n(Q)}{2^{-k}} \rho(x)  \qquad \te{for all } x \in Q.
\end{equation}
Let $\mu_k$ be the measure with density $\rho_k$. Note that for all $Q \in \mathcal F_k$, and all $j \geq k$,  $\mu_j(Q) =\mu_k(Q) = \nu_n(Q)$.
Since by construction $\nu(Q) = 2^{-k}$, $n \nu_n(Q)$ is a binomial random variable with $n$ trials and probability of success for each trial of $p = 2^{-k}$. 
Fix $\alpha >2$ and let 
$$k_n:=\log_2\left(  \frac{ n}{10 \alpha \ln n} \right). $$
Consider $k \in \N$ with $k \leq k_n$. Using Bernstein's inequality \eqref{bernstein} with $t = \frac{p}{2}$ we obtain
\begin{align}
\begin{split}
P\left( \left| \nu_n(Q) - \frac{1}{2^k} \right| \geq \frac{1}{2^{k+1}} \right) &  \leq 2 \exp \left(  - \frac{\frac12 \cdot \frac14 n p^2}{p(1-p) + \frac13 \cdot \frac12  p}  \right) \\
& \leq 2 \exp \left( - \frac{1}{10} np \right) \\
& \leq 2 \exp \left( - \frac{1}{10} n \frac{10 \alpha \ln n}{n} \right) \\
& = 2 n^{-\alpha}.
\end{split}
\label{IneqAuxMatching}
\end{align}
Since the probability of the union of events is less or equal to the sum of the probability of the events, we obtain
\begin{equation*}
P\left( \max _{Q \in \mathcal F_k} \left| \nu_n(Q) - \frac{1}{2^k} \right| \geq \frac{1}{2^{k+1}} \right) 
\leq 2^k  2 n^{- \alpha} . 
\end{equation*}
Summing over all $k  \leq k_n$, we deduce that with probability at least $1 - n^{-\alpha/2}$,
\begin{equation} \label{rhokbd}
\frac{1}{2 \lambda} \leq \rho_k \leq \frac{3 \lambda }{2} \qquad \te{on } (0,1)^d,
\end{equation}
for every $k \leq k_n$.

Let $Q \in \mathcal F_k$ and let $Q_1, Q_2 \in \mathcal F_{k+1}$ be the sub-boxes of $Q$. Let $m = n \nu_n(Q)$. Since $\nu(Q_1) = 2^{-(k+1)} = \frac12 \nu(Q)$ then, $m \frac{\nu_n(Q_1)}{\nu_n(Q)} \sim \te{Bin}(m,\frac{1}{2})$ given $\nu_n(Q)$. Using Chernoff's bound \eqref{chernoff} and \eqref{IneqAuxMatching}, we deduce that
\begin{equation*}
P\left( \left|  \frac{\nu_n(Q_1)}{\nu_n(Q)}  - \frac{1}{2} \right| \geq \sqrt{\frac{\alpha 2^k \ln n}{n}}  \right)   \leq 4 n^{-\alpha}. 
\end{equation*}
Using the previous inequality, \eqref{T5rhok} and a union bound, we conclude that
\begin{equation*} 
P\left(  \sup_{x \in (0,1)^d} \left|  \frac{\rho_{k+1}(x)}{\rho_k(x)} - 1 \right| \geq 2 \sqrt{\frac{ \alpha 2^k \ln n}{n}} \right)  \leq  2^k4 n^{-\alpha}. 
\end{equation*}
Summing over all $k  \leq k_n$, we deduce that with probability at least $1 -  n^{-\alpha/2}$,
\begin{equation} \label{rhokkp1}
\sup_{x \in (0,1)^d} \left|  \frac{\rho_{k+1}(x)}{\rho_k(x)} - 1 \right| \leq 2 \sqrt{\frac{ \alpha 2^k \ln n}{n}}
\end{equation}
for every $k \leq k_n$.

Notice that for all $Q \in \mathcal F_k$, and all $j \geq k$,  $\mu_{j}(Q) = \mu_k(Q)= \nu_n(Q)$. Then,
\begin{equation}
d_\infty(\mu_k, \mu_{k+1})  \leq \max_{Q \in \mathcal{F}_k} d_\infty(\mu_k \llcorner_Q , \mu_{k+1}\llcorner_Q ),
\label{MatchinLinifinitySubcubes}
\end{equation}
and 
\begin{equation}
d_\infty(\mu_k, \nu_n)  \leq \max_{Q \in \mathcal{F}_k} d_\infty(\mu_k \llcorner_Q , \nu_n\llcorner_Q ) \leq \max_{Q \in \mathcal{F}_k} \diam(Q)  \leq C(\lambda) \frac{1}{2^{k/d}} ,
\label{MatchinLinifinitySubcubesEmpirical}
\end{equation}
where $C(\lambda)$ is a constant only depending on $\lambda$; the last inequality in the previous expression obtained from Lemma \ref{AspectRatioLemma} and from the fact that $\nu(Q) = 2^{-k}$.

Using estimates \eqref{rhokbd} and \eqref{rhokkp1}
\[ \| \rho_k - \rho_{k+1}\|_{L^\infty((0,1)^d)}  \leq \| \rho_k \|_{L^\infty((0,1)^d)} \left\| \frac{\rho_{k+1}}{\rho_k} -1 \right\|_{L^\infty((0,1)^d)}  \leq 2 \lambda \left( \alpha 2^k \frac{\ln n}{n} \right)^{1/2}, \]
with probability at least $1 - n^{- \alpha/2}$.
Hence from Lemma \ref{AspectRatioLemma} and remark \ref{LIninfinityEstimateRectangle}, we deduce that for all $Q \in \mathcal{F}_k$ 
$$ d_\infty(\mu_k | _Q, \mu_{k+1} \llcorner_Q ) \leq  C(\lambda,d) \diam(Q) \left( \alpha 2^k \frac{\ln n}{n} \right)^{1/2} \leq C(\lambda,d)\frac{1}{2^{k/d}}  \left( \alpha  2^k \frac{\ln n}{n} \right)^{1/2}.  $$

Using \ref{MatchinLinifinitySubcubes} and the previous inequalities, we conclude that except on a set with probability $O(n^{-\alpha/2})$, for every $k=0, \dots,k_n$ 
$$  d_\infty(\mu_k, \mu_{k+1}) 
\leq C \frac{1}{2^{k/d}} \left( 2^k \frac{\ln n}{n} \right)^{1/2},  $$
for some constant $C$ depending only on $\lambda$, $\alpha$ and $d$. From the triangle inequality and \eqref{MatchinLinifinitySubcubesEmpirical}, we obtain 

\begin{align*}
 d_\infty(\nu,\nu_n ) & \leq \sum_{k=1}^{k_n}  d_\infty(\mu_{k-1}, \mu_k) + d_\infty(\mu_{k_n},\nu_n)  
\\ & \leq  C \left( \sum_{k=1}^{k_n  } \frac{1}{2^{k/d}}  \left( \alpha 2^k \frac{\ln n}{n} \right)^{1/2} + \frac{\ln n^{1/d}}{n^{1/d}} \right).
\end{align*}

Given that $d\geq 3$, the sum in the previous expression is $O(\frac{\ln n^{1/d}}{n^{1/d}})$. In summary, except on a set with probability $O(n^{-\alpha/2})$
$$ d_\infty(\nu, \nu_n)  \leq C \frac{\ln n^{1/d}}{n^{1/d}}, $$ 
where $C$ is a constant that depends on $\alpha$, $\lambda$ and $d$ only.
\end{proof}

\subsection{The matching results for $(0,1)^2$.}
\label{SecMatchingBoxdeq2}
Now we prove Theorem \ref{main} for $D=(0,1)^2$. We actually state and prove a stronger result which is in agreement with the result by Talagrand in \cite{TalagrandGenericChain}. The improvement with respect to the statement of Theorem \ref{main}, has to do with the speed of decay of the tail probability of the transportation distance. Theorem \ref{main} is an immediate consequence of the following.

\begin{theorem}
Suppose that $\rho:(0,1)^2 \rightarrow (0,\infty)$ is a density function satisfying 

\begin{equation}
  \frac{1}{\lambda} \leq \rho \leq \lambda  
  \label{lambdaBoundRhod=2}
\end{equation}

for some $\lambda>1$. Let $X_1, \dots, X_n$ be i.i.d samples from $\rho$ and denote by $\nu_n$ the empirical measure

$$ \nu_n := \frac{1}{n}\sum_{i=1}^{n}\delta_{X_i} . $$

Then, there is a constant $L >0$ depending only on $\lambda$, such that except on a set with probability $L \exp(-(\ln(n)^{3/2}  )/L  )$, we have

$$ d_\infty(\nu , \nu_n) \leq L \frac{\ln(n)^{3/4}}{n^{1/2}}. $$ 
\label{MatchingCubed=2}
\end{theorem}

In order to match the empirical measure $\nu_n$ with the measure $\nu$, we consider a partition of $(0,1)^2$ into $n$ rectangles $Q_1, \dots, Q_n$, each of which has $\nu$-measure equal to $1/n$. We then look for a bijection between the set of points $X_1, \dots, X_n$  and the set $\left\{Q_1, \dots, Q_n \right\}$, in such a way that every data point is matched to a nearby rectangle. Note however, that in order to guarantee that all points within a rectangle are close to the corresponding data point we should be able to control the diameter of all the $Q_i$s. This is is important since we want to obtain estimates on $d_\infty(\nu, \nu_n)$. With a slight modification to the construction preceding Lemma \ref{AspectRatioLemma} we obtain the following.

\begin{lemma} \label{reg_part}
Let $\rho: (0,1)^2 \rightarrow (0,\infty)$ be a density function satisfying \eqref{lambdaBoundRhod=2}, and let $\nu$ be the measure $d \nu = \rho dx$. Then, for any $n \in \N$ there exists a collection of rectangles $\{Q_i \::\: i =1, \dots, n\}$ that partitions $[0,1]^2$, such that the aspect ratio of all rectangles is less than $3 \lambda^2$ and their volume according to $\nu$ is $1/n$. In particular, for every $Q_i$
\begin{equation}
\diam(Q_i) \leq \frac{C(\lambda)}{\sqrt{n}},
\label{DiameterQi}
\end{equation}
where $C(\lambda)$ is a constant only depending on $\lambda$.
\end{lemma}

The task now is to show that with high probability we can indeed find a matching between the points $X_1, \dots, X_n$ and the rectangles $Q_1, \dots, Q_n$, in such a way that every point is close to its matched rectangle. When $\rho \equiv 1$, the previous statement is directly related to the result of Leighton and Shor \cite{LeightonShor}. The proof of Leighton and Shor depends on discrepancy estimates over all regions $R$ formed by squares from a suitable regular grid  $G'$  defined on $D$. By discrepancy we mean the difference between $\nu(R) $ and $\nu_n(R)$ for a given region $R$. Obtaining a uniform bound on the discrepancy over all regions $R$ can be interpreted as obtaining probabilistic estimates on the supremum of a stochastic process indexed by the mentioned class of regions $R$. 
A conceptually clear and efficient proof of this matching result, based on obtaining upper bounds of stochastic processes,  was presented by Talagrand  \cite{TalagrandGenericChain, TalagrandNewBook}. In order to prove Theorem \ref{MatchingCubed=2} we follow the framework of Talagrand and start by stating a general result on obtaining bounds on the supremum of more general stochastic processes (Section 1 in \cite{TalagrandGenericChain}).

Let $(Y,d)$ be an arbitrary metric space. For $n \in \N $ define,
$$ e_n(Y,d)= \inf \sup_{y \in Y} d(y,Y_n),  $$
where the infimum is taken over all subsets $Y_n$ of $Y$ with cardinality less than $2^{2^n}$.  Let $\left\{ A_n \right\}_{n \in \N}$ be a sequence of partitions of $Y$. This sequence of partitions is called \emph{admissible} if it is increasing (in the sense that for every $n$, $A_{n+1}$ is a refinement of $A_n$) and it is such that the cardinality of $A_n$ is no bigger than $2^{2^n}$. For a given $y \in Y$ and  $\left\{ A_n \right\}_{n \in \N}$ admissible, $A_n(y)$ represents the unique set in $A_n$ containing $y$.  For an $\alpha >0$, consider
$$ \gamma_{\alpha}(Y,d)= \inf \sup_{y\in Y} \sum_{n \geq 0} 2^{n /\alpha} \diam(A_n(y)),  $$
where $\diam(A_n(y))$ represents the diameter of the set $A_n(y)$ using the distance function $d$ and where the infimum is taken over all $\left\{A_n \right\}_{n \in \N}$ admissible sequences of partitions of $Y$. With these definitions we can now state Theorem 1.2.9 in \cite{TalagrandGenericChain}. 
\begin{lemma}
Let $Y$ be a set and let $d_1,d_2$ be two distance functions defined on $Y$. Let $\left\{Z_y \right\}_{y \in Y}$ be a stochastic process satisfying: for all $y,y' \in Y$ and all $u>0$
\begin{equation}
\mathbb{P}\left( |Z_{y}-Z_{y'}| \geq u \right) \leq 2 \exp \left( - \min ( \frac{u^2}{d_2(y,y')^2} , \frac{u}{d_1(y,y')} )\right),
\label{TalagrandsCondition}
\end{equation}
and also $\mathbb{E}[Z_y]=0$ for all $y \in Y$. Then, there is a constant $L>0$ large enough, such that for all $u_1, u_2 >0$ 
\begin{equation}
\mathbb{P}( \sup_{y \in Y} |Z_{y}-Z_{y_0}| \geq L ( \gamma_1(Y,d_1)+ \gamma_2(Y,d_2)) + u_1 D_1 +u_2 D_2  ) \leq L \exp(-\min\left\{u_2^2,u_1 \right\}),
\label{TalagrandsTailEstimate}
\end{equation}
where $D_1= 2 \sum_{n \geq 0} e_n(Y,d_1)$ and $D_2=2 \sum_{n \geq 0} e_n(Y,d_2)$.
\label{TalagrandsTheorem}
\end{lemma}
One of the consequences of the previous lemma is the following: in order to prove a tail estimate of the supremum of the stochastic process $\left\{Z_y \right\}_{y \in Y}$, like the one in \eqref{TalagrandsTailEstimate}, one needs to do two things. First, estimate the quantities $\gamma_1(Y,d_1)$,  $\gamma_2(Y,d_2)$, $D_1$ and $D_2$. Note that these quantities depend only on the distances $d_1,d_2$ and hence are not a priori related to the process $\left\{ Z_y \right\}_{y \in Y}$. Secondly, relate the stochastic process $\left\{ Z_y \right\}_{y \in Y}$ with the distances $d_1,d_2$ by establishing condition \eqref{TalagrandsCondition}.

We are now ready to prove Theorem \ref{MatchingCubed=2}. As mentioned earlier, this result is an adaptation of the proof by Talagrand of Leighton and Shor theorem. We sketch some of the main steps in the proof by Talagrand and give the details on how to generalize it to non-constant densities.

\begin{proof}[Proof of theorem \ref{MatchingCubed=2} ] 
In what follows $L>0$ is a constant that may increase from line to line. 

\emph{Discrepancy estimates.} Let $l_1$ be the largest integer such that $ 2^{-l_1} \geq \frac{(\ln(n))^{3/4}}{\sqrt{n}}$. Consider $G$ to be the regular grid of mesh $2^{-l_1} $ given by
\begin{equation}
G=\left\{ (x_1, x_2) \in [0,1]^2 \: ; \; 2^{l_1}x_1 \in \N \: \te{ or } \: 2^{l_1}x_2 \in \N \right\}
\label{GridG}
\end{equation}
A \emph{vertex} of the grid $G$  is a point $(x_1,x_2)$ in $[0,1]^2$ such that $2^{l_1}x_1 \in \N$ and $2^{l_1}x_2 \in \N$. A \emph{square} of the grid $G$ is a square of side length equal to $2^{-l_1}$ and whose edges belong to $G$. The edges are included in the squares.  

For a given vertex $w$ of $G$ and a given integer $k$, consider $\mathcal{C}(w,k)$ the set of simple closed curves that lie on $G$ which contain the vertex $w$ and have length $l(C) \leq 2^k$. Note that every closed simple curve $C$ in $\R^2$ divides the space into two regions, one of which is bounded; this later one is called the interior of the curve $C$ and is denoted by  $C^{\circ}$. For $C, C' \in \mathcal{C}(w,k)$ set $d_1(C,C')=1$ if $C \not = C'$ and $d_1(C,C')=0$ if $C=C'$. Also set  $  d_2(C,C')= \sqrt{n} \| \chi_{C^{\circ}} - \chi_{C'^\circ}  \|_{L^2(D)}$. 

\emph{Claim 1:} For a given $w$ of $G$ and a given integer $ k$ with $k \leq l_1 +2$, there exists $L>0$ large enough such that with probability at least $1- L \exp(-\ln(n)^{3/2}  /L  )$
\begin{equation}
\sup_{C \in \mathcal{C}(w,k)} |  \sum_{i \leq n}\left( \chi_{C^\circ}(X_i) -  \nu(C^\circ) \right) | \leq L2^k \sqrt{n}(\ln(n))^{3/4}.
\end{equation}

To prove the claim, the idea is to study the supremum of the stochastic process $\left\{ Z_C \right\}_{C \in \mathcal{C}(w,k) }$ where
$$ Z_C := \frac{1}{L} \sum_{i \leq n}   \left( \chi_{C^\circ}(X_i) -  \nu(C^\circ) \right). $$
For fixed  $C, C' \in \mathcal{C}(w, k)$ one can write the difference $Z_C -Z_{C'}$ as
$$  Z_C - Z_{C'}=  \sum_{i \leq n} Z_i,  $$
where $ Z_i = \frac{1}{L} \left(\chi_{C^\circ}(X_i) - \chi_{C'^\circ}(X_i)   - \nu(C^\circ) + \nu(C'^\circ) \right) $. The random variables $\left\{ Z_i\right\}_{i \leq n}$  are independent and identically distributed with mean zero, they satisfy $|Z_i| \leq \frac{2}{L}$ and furthermore, their variance $\sigma^2$ is bounded by
$$  \sigma^2 \leq \frac{1}{L^2} \mathbb{E}\left[ |\chi_{C^\circ}(X_i) - \chi_{C'^\circ}(X_i)|^2 \right]  \leq \frac{\lambda}{L^2} \| \chi_{C^\circ}-\chi_{C'^\circ} \|_{L^2(D)}^2. $$
Using Bernstein's inequality and choosing $L>0$ to be large enough, we obtain
\begin{equation*}
\mathbb{P}\left( \left| Z_C - Z_{C'} \right| \geq u \right) \leq 2 \exp \left(  - \frac{u^2}{   n \| \chi_{C^\circ}-\chi_{C'^\circ} \|_{L^2(D)}^2   + u}    \right) = 2 \exp \left( - \min ( \frac{u^2}{d_2(C,C')^2} , \frac{u}{d_1(C,C')} ) \right).
\end{equation*}
In the proof of proposition 3.4.3 in Talagrand, the estimates  $\gamma_1(\mathcal{C}(w,k), d_1 ) \leq L 2^{k}\sqrt{n}$ , $\gamma_2(\mathcal{C}(w , k),d_2)  \leq  L 2^k \sqrt{n}  (\ln(n))^{3/4} $, $D_1 \leq 2(k+l_1 +1)$ and $D_2\leq  L 2^{k+1}\sqrt{n}$ are established. Setting $u_1= (\ln(n))^{3/2}$ and $u_2= (\ln(n))^{3/4}$ one can use Lemma \ref{TalagrandsTheorem} ( with $Y= \mathcal{C}(w, k)$, $d_1$, $d_2$ as above and $y_0=\left\{w \right\}$ ) to prove the claim.

Considering all possible vertices $w$ of $G$ and all possible integers $k$ with $-l_1 \leq k \leq l_1 +2$. It is a direct consequence of Claim 1 above that with probability at least $1- L \exp(-(\ln(n)^{3/2}  )/L  )$, 
\begin{equation}
\sup_{C } |  \sum_{i \leq n}\left( \chi_{C^\circ}(X_i) -  \nu(C^\circ) \right) | \leq Ll(C) \sqrt{n}(\ln(n))^{3/4},
\label{eventDiscrepancy}
\end{equation}
where the supremum is taken over all $C$ closed, simple curves on $G$. See the proof of Theorem 3.4.2 in \cite{TalagrandGenericChain}.  We denote by $\Omega_n$ the event for which \eqref{eventDiscrepancy} holds.

\emph{Enlarging Regions. }Consider an integer $l_2$ with $l_2 < l_1$.   We consider $G'$ the grid defined as in \eqref{GridG} but with mesh size $2^{-l_2}$. Note that in particular $G' \subseteq G$. Let $R$ be a union of squares of the grid $G'$. One can define $R'$ to be the region formed by taking the union of all the squares in $G'$ with at least one side contained in $R$.  With no change in the proof of Theorem 3.4.1 in \cite{TalagrandGenericChain}, it follows from the discrepancy estimates obtained previously that in the event $\Omega_n$ one has
\begin{equation}
 \nu(R') \geq  \nu_n(R)
\label{IneqRandR'}
\end{equation}
for all regions $R$ formed with squares from $G'$, provided that $2^{-l_2} \geq \frac{2^6 L}{\sqrt{n}}(\ln(n) )^{3/4}$. 

What this is saying is that given the discrepancy estimates obtained previously, in the event $\Omega_n$, for any region $R$ formed by taking the union of squares in $G'$, one can enlarge $R$ a bit to obtain a region $R'$ in such a way that the area of the enlarged region $R'$ according to $\nu$ is greater than the area of the original region $R$ according to $\nu_n$. It is worth remarking that the restriction to the number $2^{-l_2}$ (the mesh size of $G'$), for this to be possible, coincides with the scaling for the transportation cost we are after. 

\emph{Matching between rectangles and data points.} We choose $l_2$ to be the largest integer satisfying $2^{-l_2} \geq \frac{2^6 L}{\sqrt{n}}(\ln(n) )^{3/4}$. Consider $\left\{Q_1, \dots, Q_n \right\}$ the rectangles constructed from Lemma \ref{reg_part}. For $i \in \left\{1, \dots, n \right\}$ let  $B_i= \left\{  j \leq n \: : \:  \dist(X_i ,Q_j ) \leq 2 \sqrt{2} \cdot 2^{-l_2}  \right\}$.

\emph{Claim 2:} In the event $\Omega_n$,  there is a bijection $\pi : \left\{1, \dots, n \right\} \rightarrow \left\{1 , \dots, n \right\}$ with $\pi(i) \in B_i$ for all $i$. 

By the Hall marriage lemma, to prove this claim it is enough to prove that for every $I \subseteq \left\{X_1, \dots, X_n \right\}$, the cardinality of $\cup_{i \in I} B_i$ is greater than the cardinality of $I$.  Fix $I \subseteq \left\{1, \dots, n \right\}$ and denote by $R_I$ the region formed with the squares of $G'$ that contain at least one of the points $X_i$ with $i \in I$. Now, take $J=\left\{ j \leq n \: : \: Q_j \cap (R_I)' \not = \emptyset \right\}$, then, $ J \subseteq \cup_{i \in I} B_i $. From the properties of the boxes $Q_i$ and from \eqref{IneqRandR'} it follows that  $\# \cup_{i \in I} B_i \geq \# J =n \nu(\cup_{j \in J} Q_j ) \geq n \nu((R_I)') \geq \# I $. This proves the claim. 

Finally, we construct a transportation map $T_n$ between $\nu$ and $\nu_n$. Indeed, for $x$ in $Q_i $, set $T_n(x)= X_{\pi^{-1}(i)}$. From the properties of the boxes $Q_i$, it is straightforward to check that $T_{n\sharp}\nu= \nu_n$ and that  $\|T_n - Id\|_{L^\infty(D) } \leq L \frac{(\ln(n))^{3/4}}{\sqrt{n}}$ due to the estimate on the diameter of the rectangles $Q_i$ in \eqref{DiameterQi}. 
\end{proof}
\section{The matching results for general $D$.}
\label{SecMatchingD}
The goal of this section is to prove the optimal bounds on matching for all open, connected, bounded domains $D$ with Lipschitz boundary.  In order to achieve this, we first prove Theorem \ref{PropositionLinfinityWinfinityGeneralDomain} for general domains $D$. It is useful to consider first a class of domains $D$ which are \emph{well partitioned}.
\begin{definition}
Let $D \subseteq \R^d$. We say that $D$ satisfies the (WP) property with $k$ polytopes if $D$ is an open, bounded and connected set  and is such that there exists a finite family of closed convex polytopes $\left\{A_i \right\}_{i=1}^{k}$ covering $D$ and satisfying: For all $i,j = 1, \dots, k$
\begin{enumerate}
\item  $\inte(A_i)  \cap D  \not = \emptyset$
\item If $i \not = j$ then  $\inte(A_i) \cap \inte(A_j) = \emptyset$.
\item $A_i \cap \overline{D}$ is bi-Lipschitz homeomorphic to a closed cube.  
\end{enumerate}
\label{DefinitionHProperty}
\label{WP}
\end{definition}

The class of domains satisfying the (WP) property is convenient for our purposes for two reasons. The first one because as we see below, in order to prove the matching results for sets with the (WP) property, we can use induction on the number of polytopes. The second reason, has to do with the fact that the class of sets which are well partitioned contains the class of open, bounded, connected domains with smooth boundary. This is the content of the next proposition whose proof is presented in the Appendix.

\begin{proposition}
Let $D\subseteq \R^d$ be an open, bounded and connected domain with smooth boundary. Then, $D$ satisfies the (WP) property with $k$ polytopes for some $k \in \N$.
\label{HPropertySmoothDomains}
\end{proposition}

We now prove a lemma that prepares the ground for an inductive argument to be used in the proof of the matching results for domains with the (WP) property.

\begin{lemma}
Suppose that $D$ is a domain which satisfies hypothesis (WP) with $k$ polytopes ($k>1$). Let $\left\{ A_i\right\}_{i=1}^{k}$ be associated polytopes. Then there exists $j$ such that $D' := D \setminus A_j $ is connected. 
\label{LemmaRemoveRectangle}
\end{lemma}
\begin{proof}
We say that $A_l \sim A_m$  if  $\relint (\partial A_m) \cap \relint(\partial A_l) \cap D \not = \emptyset$,  where $\relint(\partial A_i)$ is the union of the relative interiors of the facets of $A_i$ ( $(d-1)$-dimensional faces).  
This relation induces a graph $G=(V,E)$ where the set of nodes $V$ is the set of polytopes $A_i$ and where an edge between $A_m$ and $A_l$ ($m \not = l$) belongs to the graph if and only if $A_m \sim A_l$.  We claim that $G$ is a connected graph.

Indeed, consider $m\not = l$. We want to show that there exists a path in the graph $G$ connecting $A_m$ with $A_l$.  For this purpose consider $x \in \inte(A_m) \cap D$ and  $y \in \inte(A_l) \cap D$. Denote by $C$ the union of all the ridges ($(d-2)$-dimensional faces) of all the polytopes $A_i$.  Given that $C$ is the union of finitely many $(d-2)$-dimensional objects in $\R^d$, we conclude that $D \setminus C$ is a connected open set and as such it is path connected. Since $x \in \inte(A_m) \cap D$  and $y \in \inte(A_l) \cap D$, in particular $x,y \in D\setminus C$ and so there exists a continuous function $\gamma: [0,1] \rightarrow D \setminus C$ such that $\gamma(0)=x$ and $\gamma(1)=y$. Let $A_{i_0} , A_{i_1} , \dots, A_{i_N}$  be the polytopes visited by the path $\gamma$ in order of appearance; this list satisfies $A_{i_s} \not = A_{i_{s+1}}$ for all $s$, $A_{i_0}= A_m $ and $A_{i_N}= A_l $. Now, note that for any given $s$, the path $\gamma$ intersects $\partial A_{i_s} \cap \partial A_{i_{s+1}} $ at a point which belongs to the relative interior of a facet ($d-1$ dimensional face) of $A_{i_s}$ and of $A_{i_{s+1}}$; this because  $\gamma $ lies in $D \setminus C$. From this fact we conclude that  $A_{i_s} \sim A_{i_{s+1}}$ and hence there is a path in $G$ connecting $A_m$ and $A_l$. This proves that $G$ is connected.

From the fact that $G$ is connected, we deduce that it has a spanning tree $G'$. That is, there exists a subgraph $G'$ of $G$ which is a tree and includes all of the vertices of $G$. Let $A_j$ be a leave of the spanning tree $G'$. It is now straightforward to show that $A_j$ is the desired polytope from the statement.
\end{proof}

\begin{remark}
Consider $D$ and $A_j$ as in the statement of Lemma \ref{LemmaRemoveRectangle}. Then $D' := D \setminus A_j$ satisfies the property (WP) with $(k-1)$ polytopes and  $D'':= D \cap A_j$ satisfies the property (WP) with one polytope.
\label{RemarkInduction}
\end{remark}

Let $A_j$ be the polytope as in statement of Lemma \ref{LemmaRemoveRectangle}. Note that there exists $i \not = j$ such that $\relint( \partial A_i) \cap \relint(\partial A_j) \cap D \not = \emptyset$; we denote this polytope by $\tilde{A}_j$. Let $\tilde{x} \in \relint(\partial \tilde{A}_j) \cap \relint(\partial A_j) \cap D$. Note that necessarily $F:= \relint(\partial \tilde{A}_j) \cap  \relint(\partial A_j) $ is contained in a hyperplane and hence we can consider $e$ a unit vector which is orthogonal to $F$. Take $r>0$ such that $B(\tilde{x},r) \subseteq  \inte(( \tilde{A}_j \cup A_j) \cap D )$. Let  $z_{1} := \tilde{x}+ r e$ and let $z_{-1}:= \tilde{x}-re$. Without the loss of generality we can assume that $z_1 \in \inte(\tilde{A}_j)$. Denote by $C_1$ the set of points of the form $t z_1 + (1-t) y $ where $t \in [0,1] $ and $y \in B(\tilde{x},r)\cap F$, similarly, denote by $C_{-1}$ the set of points of the form $tz_{-1} +(1-t)y$ where $t \in [0,1]$ and where  $y \in B(\tilde{x},r)\cap F$. Let $z_{-1/2}:= \tilde{x}-\frac{r}{2}e$ and consider the set $C_{-1/2}$ defined analogously to the way $C_1$ and $C_{-1}$ are defined. We can think of $C_1$ and $C_{-1}$ as \emph{gates} connecting the sets $D'= D \setminus A_j$ and $D''= D \cap A_j$. We illustrate the  construction on Figures \ref{fig:peanut3D} and \ref{fig:Zoom}.  
\begin{figure}
\begin{minipage}[t]{7cm}
\resizebox{7cm}{!}{\includegraphics{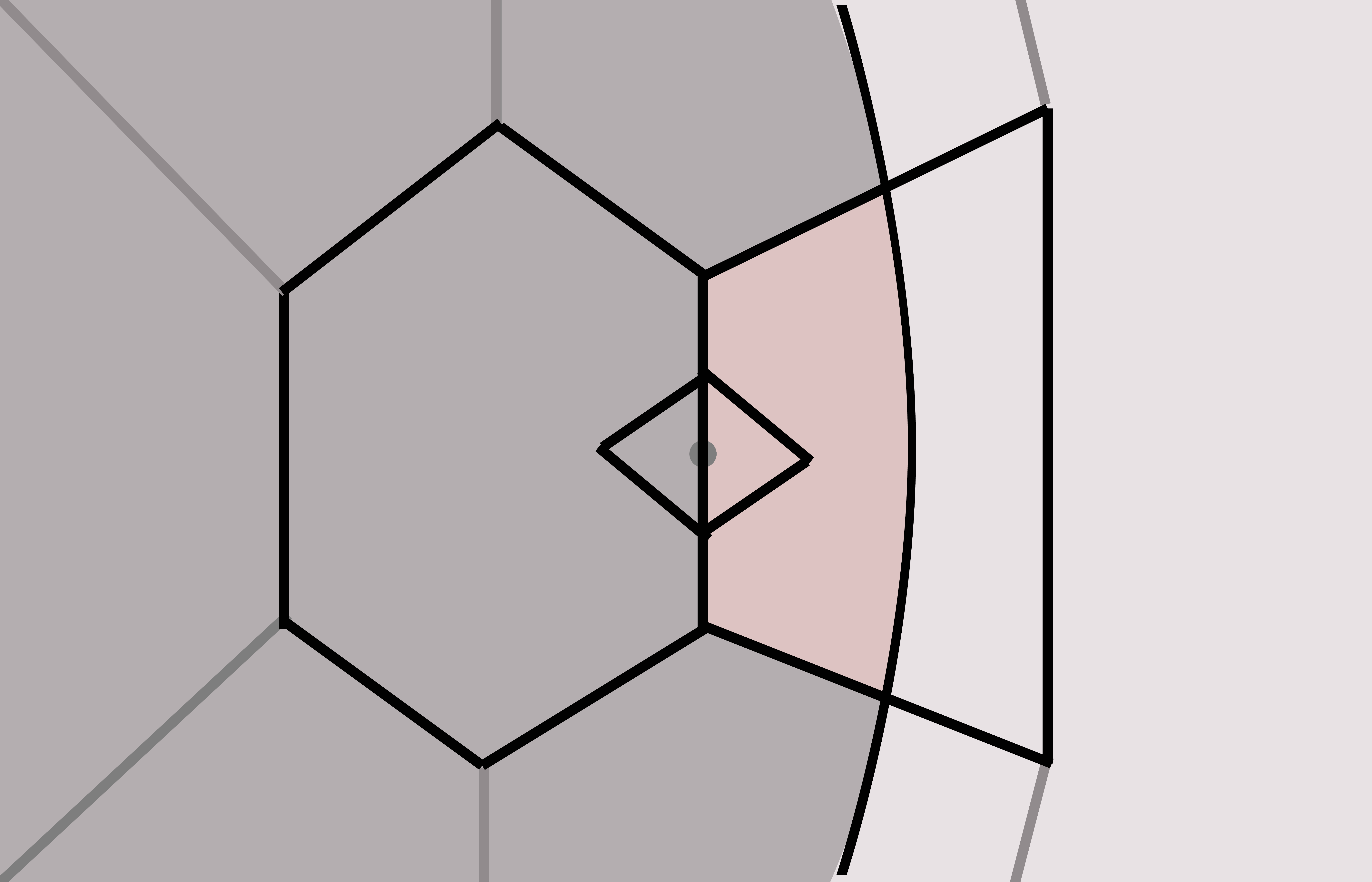}}
\put(-95,4){\LARGE $D$}
\put(-88,78){\large $D''$}
\put(-95,60){\small $\tilde{x}$}
\put(-140,65){\Large $\tilde A_j$}
\put(-64,78){\Large ${A}_j$}
\caption{Polytope $A_j$ with neighbor $\tilde{A}_j$.}
\label{fig:peanut3D}
\end{minipage}
\hspace*{8pt} 
\begin{minipage}[t]{7cm}
\resizebox{7cm}{!}{\includegraphics{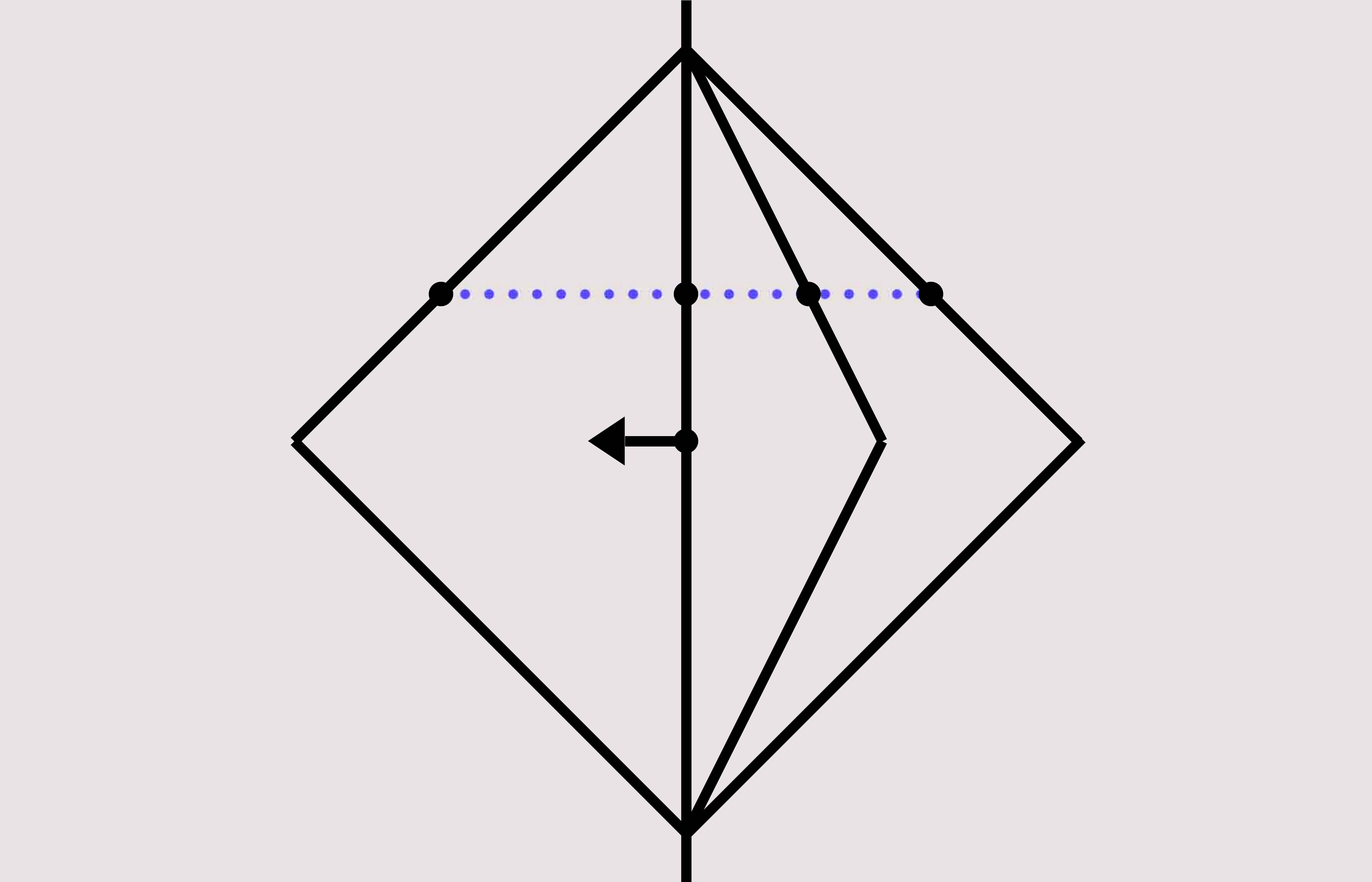}}
\put(-127,49){\Large $C_1$}
\put(-98,49){\large$C_{\te{-} \frac12}$}
\put(-75,49){\Large $C_{\te{-}1}$}
\put(-97,62){\large $\tilde{x}$}
\put(-167,62){\large $z_1$}
\put(-40,62){\large $z_{\te{-}1}$}
\put(-108,70){$e$}
\put(-140,91){$y_1$}
\put(-106,91){$y$}
\put(-63,91){$y_{-1}$}
\put(-156,64){\circle*{3}}
\put(-43,64){\circle*{3}}
\caption{Gate, enlarged.}
\label{fig:Zoom}
\end{minipage}
\end{figure}

We claim that there is a function $\psi : D'' \cup C_1 \rightarrow D'' $ which is a bi-Lipschitz homeomorphism. In fact, for a given point $y \in F \cap B(\tilde x,r)$ consider the line with direction $e$ passing trough the point $y$. This line intersects $\partial C_1$, at the points $y$ and $y_1$, it intersects $\partial C_{-1}$ at the points $y$ and $y_{-1}$ and finally it intersects $\partial C_{-1/2}$ at the points $y$ and $y_{-1/2}$. We set $\psi(y_{1}):= y $, $\psi (y) := y_{-1/2}$ and $\psi(y_{-1}):=y_{-1}$. On the segments  $[y_{-1},y]$, $[y,y_1]$ we define $\psi$ to be continuous and piecewise linear.  In this way we  define $\psi$ for all points in $C_{1} \cup C_{-1}$. Finally, set $\psi$ to be the identity on $D'' \setminus C_{-1}$. It is straightforward to check that $\psi$ constructed in this way is a bi-Lipschitz homeomorphism.

Now we are ready to prove Theorem \ref{PropositionLinfinityWinfinityGeneralDomain} for general domains.

\begin{proof}[Proof of Theorem \ref{PropositionLinfinityWinfinityGeneralDomain}]
\emph{Step 1:} Instead of proving the result for domains as in the statement, we first prove the result for domains $D$ satisfying the (WP) property. The proof is by induction on the number of polytopes $k$.

We remark that the constant $D(d, \lambda)$ may change (increase) from line to line in the proof.
\emph{Base case.} Suppose $k=1$. In this case there exists $\psi: \overline{D} \rightarrow [0,1]^d$ a bi-Lipschitz homeomorphism between $\overline{D}$ and the unit box. We use the map $\psi$ to obtain measures $\tilde{ \nu}_1, \tilde{\nu}_2 $  on $(0,1)^d$ by setting
$$  \tilde{\nu}_i:= \psi_{\sharp}\nu_i    \quad \te{ for } i=1,2. $$
Using the fact that $\psi$ is bi-Lipschitz, we can use the change of variables formula to deduce that $\tilde{\nu_1}$ and $\tilde{\nu}_2$ are absolutely continuous with respect to the Lebesgue measure with densities
$$ \tilde{\rho}_i(y) = \rho_i(\psi^{-1}(y)) | \det( J \psi^{-1}(y) )|    \quad \te{ for } i=1,2.  $$
 Here, $J\psi^{-1}$ represents the Jacobian matrix of $\psi^{-1}$.
 
 Using the fact that $\psi$ is bi-Lipschitz, we deduce that
$$   \frac{1}{\tilde{\lambda}} \leq \tilde{\rho}_1, \tilde{\rho}_2 \leq \tilde{\lambda}    $$
where $\tilde{\lambda} = \max\{\Lip(\psi)^d,  \Lip(\psi^{-1})^d \}$. By Theorem \ref{PropositionLinfinityWinfinityGeneralDomain} applied to the unit cube,
\begin{equation*}
 d_\infty(\tilde{\nu}_1, \tilde{\nu}_2) \leq C(\tilde{\lambda},d) \| \tilde{\rho}_1- \tilde{\rho}_2\|_{L^\infty((0,1)^d)}. 
 \label{AuxLInfinityGeneralDomain}
\end{equation*}
Consequently,
\begin{equation*}
 d_\infty(\nu_1, \nu_2) \leq \Lip(\psi^{-1}) d_\infty(\tilde{\nu}_1, \tilde{\nu}_2) \leq C \| \tilde{\rho}_1- \tilde{\rho}_2\|_{L^\infty((0,1)^d)} \leq C \| \rho_1- \rho_2\|_{L^\infty(D)}. 
 \label{AuxLInfinityGeneralDomain}
\end{equation*}
for some constant $C$ depending on $\lambda$ and $D$ only. 

\emph{Inductive Step.} Suppose that for any domain in $\R^d$ satisfying the (WP) property with $(k-1)$ polytopes the proposition is true. Let $D$ be a domain satisfying the (WP) property with $k$ polytopes and let $\rho_1, \rho_2: D \rightarrow (0,\infty)$ be functions as in the statement. By relabeling the functions if necessary, we can assume without the loss of generality that  $\int_{D'} \rho_1(x) dx - \int_{D'}\rho_2(x) dx \geq 0 $, where $D'$ is as in Remark \ref{RemarkInduction}.  Since there is more mass in $D'$ according to $\nu_1$ than according to $\nu_2$, we decide to transfer this excess of mass from the set $D'$ to the set $D''$. To achieve this, we first move the excess of mass on $D'$ to the gate $C_1$, so that we can subsequently move it to the set $D''$. In mathematical terms, we consider an intermediate distribution $d\tilde{\nu}_1= \tilde{\rho}_1 dx$ where   
\begin{equation*}
\tilde{\rho}_1(x) := 
\begin{cases}
 \rho_2(x) , \:  & \te{if }   x \in D' \setminus C_1
 \\
 \beta \rho_1(x),\: &\te{if } x \in C_{1}
\\ \rho_1(x), \: &\te{if }  x \in D'',
\end{cases}
\end{equation*}
and where
$$\beta =  \frac{\int_{D'}( \rho_1(x) - \rho_2(x)  ) dx   + \int_{C_1} \rho_2(x)dx  }{\int_{C_1} \rho_1(x)dx };$$ 
the idea is to compare $\nu_1$ with $\tilde{\nu}_1$ and then compare $\tilde{\nu}_1$ with $\nu_2$.

First, note that there is a $\lambda'>1$ depending only on $\lambda$ and $D$ such that
$$  \frac{1}{\lambda'}   \leq \rho_1, \tilde{\rho}_1 \leq \lambda'.  $$
Since by construction $\nu_1(D')= \tilde{\nu}_1(D') $, we use Remark \ref{RemarkInduction} and the induction hypothesis to conclude that:
$$  d_\infty(\nu_1 \llcorner_{D'} , \tilde{\nu}_1\llcorner_{D'})  \leq C(\lambda', D') \| \rho_1 - \tilde{\rho}_1\|_{L^\infty(D')}= C(\lambda, D)\| \rho_1 - \tilde{\rho}_1\|_{L^\infty(D')}  , $$
where $\nu_1\llcorner_{D'}$ denotes the measure $\nu_1$ restricted to $D'$ and $\tilde{\nu}_1 | _{D'}$ the measure $\tilde{\nu}_1$ restricted to $D'$; notice that we can write $C(\lambda',D')= C(\lambda,D)$ because $\lambda'$ depends on $\lambda$ and $D$ only. An immediate consequence of the previous estimate is that
\begin{equation}
 d_\infty(\nu_1 , \tilde{\nu}_1)  \leq C(\lambda, D) \| \rho_1 - \tilde{\rho}_1\|_{L^\infty(D)}.   
 \label{auxLinfinityGeneralDomain}
\end{equation}
Given the definition of $\beta$, it is straightforward to show that 
$$   \| \rho_1 - \tilde{\rho}_1     \|_{L^\infty(D)} \leq C(\lambda, D)  \| \rho_1- \rho_2\|_{L^\infty(D)} $$
for some constant $C(\lambda, D)$ only depending on $D$ and $\lambda$. The previous inequality combined with  \eqref{auxLinfinityGeneralDomain} gives:
$$  d_\infty(\nu_1 , \tilde{\nu}_1) \leq C(\lambda , D)  \| \rho_1- \rho_2\|_{L^\infty(D)}.   $$

Now we compare $\tilde{\nu}_1$ with $\nu_2$. First of all note that $ \tilde{\nu}_1(D''_1) =\nu_2(D''_1) $ , where $D''_1:=D'' \cup C_1$. From the discussion proceeding Remark \ref{RemarkInduction} we know that $D''_1$ is bi-Lipschitz homeomorphic to the set $D''$ which in turn is bi-Lipschitz homeomorphic to the unit box. Thus, $D''_1$ is bi-Lipschitz homeomorphic to the unit box and hence proceeding as in the base case, we conclude that
$$ d_\infty(\tilde{\nu}_1\llcorner_{D''_1}, \nu_2\llcorner_{D''_1}  ) \leq C(\lambda, D) \| \tilde{\rho}_1 -\rho_2  \|_{L^\infty(D''_1)}  $$
and consequently
$$ d_\infty(\tilde{\nu}_1, \nu_2  ) \leq C(\lambda, D) \| \tilde{\rho}_1 -\rho_2  \|_{L^\infty(D)}.   $$

A straightforward computation shows that $\| \tilde{\rho}_1 - \rho_2\|_{L^\infty(D)} \leq C(\lambda, D)\| \rho_1 - \rho_2\|_{L^\infty(D)} $  and thus
$$ d_\infty(\tilde{\nu}_1, \nu_2  ) \leq C(\lambda, D) \|\rho_1 -\rho_2  \|_{L^\infty(D)}.   $$

Using the previous inequality, \eqref{auxLinfinityGeneralDomain} and the triangle inequality we obtain the desired result.
\medskip

\emph{Step 2:} Now consider an open, connected bounded domain $D$ with Lipschitz boundary. By Remark 5.3 in \cite{BallZ} there exists an open set $\tilde{D}$ with smooth boundary which is bi-Lipschitz homeomorphic to $D$. In particular $\tilde{D}$ is bounded and connected. By propositions \ref{HPropertySmoothDomains} and Step 1, the result holds for $\tilde{D}$. Proceeding as in the base case in Step 1 and using the fact that $D$ and $\tilde{D}$ are bi-Lipschitz homeomorphic we obtain the desired result.  
\end{proof}

Now we are ready to prove Theorem \ref{main}.

\begin{proof}[Proof of Theorem \ref{main} ] 

Let us consider the function $\phi: \mathbb{N} \rightarrow (0,\infty)$, which is given by
\begin{equation}
  \phi(n)=\begin{cases}
    \frac{\ln(n)^{1/d}}{n^{1/d}}, \: & \te{if } \: d\geq 3 \smallskip \\
    \frac{\ln(n)^{3/4}}{n^{1/2}} , \: & \te{if } \: d=2.
\end{cases}
\end{equation}

\emph{Step 1.} We first prove the result for domains $D$ satisfying the (WP) property. The proof is by induction on $k$, the number of polytopes used in the definition of the property (WP). In what follows $C$ may change from line to line, but always represents a constant that depends only on $\lambda$ and $D$. Furthermore, since the probability that no sample point belongs to a boundary of one of the $k$ polytopes is zero, we assume 
without the loss of generality that no sample point belongs to the boundary of any of the polytopes considered.
 
\emph{Base Case.} Suppose that $D$ is a domain satisfying property (WP) with one polytope. Then, $\overline{D}$ is bi-Lipschitz homeomorphic to the unit box. That is, there exists a bi-Lipschitz mapping $\psi: \overline{D} \rightarrow [0,1]^d$. Given a density $\rho: D \rightarrow (0,\infty)$ satisfying \eqref{DensityBound}, we define measure  $\tilde{ \nu}$ on $(0,1)^d$  to be the push-forward of $\nu$ by $\psi$:
$$  \tilde{\nu}:= \psi_{\sharp}\nu. $$
Given the i.i.d. random points $X_1, \dots, X_n$ on $D$ distributed according to $\nu$ we note that
$$ \tilde{X}_i =\psi(X_i) \; \te{ for } i=1,\dots,n  $$ 
are i.i.d random points on $(0,1)^d$  distributed according to $\tilde{\nu}$.

As in the proof of Theorem \ref{PropositionLinfinityWinfinityGeneralDomain} we use the fact that $\psi$ is bi-Lipschitz to deduce that $\tilde{\nu}$ has a density $\tilde{\rho}$ satisfying
$$   \frac{1}{\tilde{\lambda}} \leq \tilde{\rho} \leq \tilde{\lambda}    $$
where $\tilde{\lambda} = \lambda \max\{\Lip(\psi)^d, \Lip(\psi^{-1})^d \}$. From Theorem \ref{main} applied to the unit cube, we know that for $\alpha > 2$, except on a set with probability $O(n^{-\alpha /2})$,
\begin{equation*}
 d_\infty(\tilde{\nu},\tilde{\nu}_n)\leq C\phi(n),
\end{equation*}
which implies
\begin{equation*}
 d_\infty(\nu,\nu_n)\leq \Lip(\psi^{-1}) d_\infty(\tilde{\nu}, \tilde{\nu}_n)  \leq C\phi(n).
\end{equation*}
where $C$ only depends on $\lambda$, $D$ and $\alpha$.

\emph{Inductive Step.} Suppose that the theorem is true for any domain in $\R^d$ satisfying the (WP) property with $k-1$ polytopes. Let $D$ be a domain satisfying the (WP) property with $k$ polytopes and let $\rho: D \rightarrow (0,\infty)$ be a density function satisfying \eqref{DensityBound}. Consider $\tilde{\rho}_n : D \rightarrow D$ the density function given by
\begin{equation}
  \tilde{\rho}_n(x)=
   \begin{cases}  
  \frac{\nu_n(D')}{\nu(D')} \rho(x), & \te{if } \: x \in D'  \smallskip \\ 
  \frac{\nu_n(D'')}{\nu(D'')} \rho(x), \; & \te{if } \: x \in D'',
\end{cases}
\end{equation}
where $D'$ and $D''$ are as in Remark \ref{RemarkInduction}. Let $\tilde{\nu}_n$ be the measure $d \tilde{\nu}_n=\tilde{\rho}_n dx$ and note that $\nu_n(D')=\tilde{\nu}_n(D')$ and $\nu_n(D'') = \tilde{\nu}(D'')$. Also, notice that
\begin{equation}
\|\rho - \tilde{\rho}_n \|_{L^\infty(D)} \leq C |\nu_n(D') - \nu(D') |,   
\label{AuxMatchingGeneral0}
\end{equation}
for some constant  $C$ that depends only on $\lambda$ and $D$.

To give some probabilistic estimates on $|\nu_n(D') - \nu(D') |$, we use Chernoff's inequality \eqref{chernoff} to conclude that \begin{equation}
 P \left( |\nu_n(D') - \nu( D' ) |   >  \sqrt{ \frac{\alpha\ln(n)}{n}} \right) \leq 2 n^{-2\alpha}.  
 \label{AuxMatchingGeneral1}
\end{equation}

Denote by $\Omega_n$ the event in which $|\nu_n(D') - \nu(D) | \leq \sqrt{ \frac{\alpha\ln(n)}{n}}$. By \eqref{AuxMatchingGeneral0} and Theorem \ref{PropositionLinfinityWinfinityGeneralDomain} (from its proof, it holds for well partitioned domains), given $\Omega_n$ we have:
\begin{equation}
 d_\infty(\nu, \tilde{\nu}_n) \leq C \frac{\ln(n)^{1/2}}{n^{1/2}}.  
\label{AuxMatchingGeneral2} 
\end{equation}
We use the fact that $\nu_n(D')=\tilde{\nu}_n(D')$ and  $\nu_n(D'')=\tilde{\nu}_n(D'')$ to estimate $d_\infty(\tilde{\nu}_n, \nu_n)$. Indeed, by the induction hypothesis, given the event $\Omega_n$, with probability at least $1- c n^{-\alpha/2}$ 
$$   d_\infty( \tilde{\nu}_n\llcorner_{D'},  \nu_{n}\llcorner_{D'}) \leq    C \phi(n) \; \te{ and } \;
   d_\infty( \tilde{\nu}_n\llcorner_{D''},  \nu_{n}\llcorner_{D''}) \leq    C \phi(n). $$
 In case the previous inequalities hold we conclude that
$$   d_\infty(\tilde{\nu}_n , \nu_{n} ) \leq \max \left\{ d_\infty( \tilde{\nu}_n\llcorner_{D'},  \nu_{n}\llcorner_{D'}) , d_\infty( \tilde{\nu}_n\llcorner_{D''},  \nu_{n}\llcorner_{D''}) \right\}  \leq C \phi(n). $$
Thus, given $\Omega_n$, with probability at least $1- c n^{-\alpha/2}$ 
$$ d_\infty(\tilde{\nu}_n,  \nu_n) \leq C \phi(n).   $$

From the previous discussion, \eqref{AuxMatchingGeneral1} and \eqref{AuxMatchingGeneral2} we conclude that with probability at least $1- c n^{-\alpha/2}$,
$$ d_\infty(\nu, \nu_n) \leq C \phi(n) + C \frac{\ln(n)^{1/2}}{n^{1/2}}  \leq C \phi(n).$$

\emph{Step 2.} To prove the theorem for an arbitrary open, connected, bounded domain $D$ with Lipschitz boundary it is enough to notice that by Remark 5.3 in \cite{BallZ} there exists an open set $\tilde{D}$ with smooth boundary which is bi-Lipschitz homeomorphic to $D$. In particular $\tilde{D}$ is bounded and connected. By Proposition \ref{HPropertySmoothDomains} the result holds for $\tilde{D}$ by Step 1. Proceeding as in the base case in Step 1 and using the fact that $D$ and $\tilde{D}$ are bi-Lipschitz homeomorphic we obtain the desired result.
\end{proof}

\subsection*{Acknowledgments}
The authors are grateful to Felix Otto and Zilin Jiang for enlightening discussions. The authors are also grateful to Michel Talagrand for letting them know of the elegant proofs of matching results in \cite{TalagrandGenericChain} and generously sharing the relevant chapters of his upcoming book \cite{TalagrandNewBook}. DS is grateful to  NSF (grant DMS-0908415).The research was also supported by NSF PIRE grant  OISE-0967140.
Authors are thankful to the Center for Nonlinear Analysis (NSF grant DMS-0635983) for its support.

\appendix
\section{Proof of Proposition \ref{HPropertySmoothDomains}}
\label{Appendix}

Consider $D$ to be a bounded open set with smooth boundary. For $\veps>0$ we denote by $\partial_\veps D$ the set of points $x \in \R^d$ with $d(x,\partial D) \leq \veps $. The fact that $\partial D$ is a smooth compact manifold implies that there exists $0<\veps_0<1$ such that for every $ x\in \partial_{\veps_0} D$ there is a unique  point $P(x)$ on $\partial D$ closest to $x$. Furthermore the function $P: x \in \partial_{2 \veps_0} D\mapsto P(x)$ is  smooth. 

For a given $z \in \partial D$ we let $\vec{n}_z$ be the unit outer normal vector to $\partial D$ at the point $z$. The fact that $\partial D$ is a smooth manifold in $\R^d$ also implies that the outer unit normal vector field changes smoothly over $\partial D$. 

We consider the signed distance function to $\partial D$,   $g : \partial_{2 \veps_0} D \longrightarrow \R$
\begin{equation}
 g(y) :=
 \begin{cases}
  \quad\! \dist(y, \partial D), \: & \te{if } \: y \in D^c 
\\ - \dist(y, \partial D), \: & \te{if } \: y \in D .
\end{cases}
\end{equation}
This function is smooth and its gradient is given by
\begin{equation}
\nabla g(y) = \vec{n}_{P(y)}.
\label{GradientSignedDistance}
\end{equation}
 We remark that for every $y \in \partial _{\veps_0} D$, $g(y) =  |y-P(y) |$  if $y \not \in D$ and $g(y)= -  |y-P(y) |$ if $y \in D$.

For a fixed $0<\veps< \veps_0$ consider the family of open balls $\left\{ B(x,\veps^2) \right\}_{x \in \partial D }$. This is an open cover of the set $\partial D$ which is compact. Hence, there exists a finite subcover $ \left\{ B(x_1, \veps^2) , \dots , B(x_N, \veps^2) \right\}$ of $\partial D$. To fix some notation, we let $\vec{n}_{i}$ be the vector $\vec{n}_{x_i}$ and we let $T_i$ be the tangent plane to $\partial D$ at the point $x_i$.  Let $V_1, \dots V_N$ be the Voronoi cells induced by the points $x_1, \dots, x_N$; that is  we let $V_i$ be the set 
$$  V_i := \left\{ y \in \R^d \: : \: |x_i - y | \leq |x_j - y|, \: \: \forall j \not =i  \right\}.  $$

Note that for every $t \in [-\veps, \veps]$ we have $P(x_i + t \vec{n}_i)=x_i$. In particular,
\begin{equation}
 |x_i + t \vec{n}_i  -  x_i | < |   x_i + t \vec{n}_i - x_j  |,   
 \label{AlongNormalDirectionVoronoi}
\end{equation}
for every $j \not = i$. Consider $\tilde{x}_i$ to be the point $\tilde{x}_i:= -\frac{\veps}{2} \vec{n}_i + x_i $ and Let  $T_i^+ := \veps \vec{n}_i + T_i $, $T_i^-:= \veps \vec{n}_i + T_i$ be the planes parallel to $T_i$ passing though the points $\veps \vec{n}_i + x_i $ and $-\veps \vec{n}_i  +x_i$ respectively. We denote by $S_i$ the closed strip delimited by the planes $T_i^{+} $ and $T_i^{-}$ and let $A_i:= V_i \cap S_i$. See Figure \ref{sunset}.

We first want to show that the region $A_i$ is contained in a circular cylinder whose axis is the line passing through the point $x_i$ with direction $\vec{n}_i$ and whose radius is small compared to $\veps$. To achieve this, for a point $y \in \R^d$ denote by $y_i$ the projection of $y$ along the line passing through $x_i$ with direction $\vec{n}_i$.

\emph{Claim 1:} For all  $0<\veps < \frac{\veps_0}{2}$ small enough,  $y \in A_i$  implies that $|y -y_i| \leq 4 \veps^{3/2}$.
 
To prove the claim suppose for the sake of contradiction that there is $y \in A_i$ with $ | y -y_i| \geq 4 \veps^{3/2}$. Since $y \in S_i$, in particular $|y_i- x_i | = \dist(y_i,\partial D ) \leq \veps$. Consider a point $\tilde{y}$ in the segment $[y,y_i]$ such that $ 4 \veps^{3/2} \geq | \tilde{y} -y_i| \geq 3 \veps^{3/2}$.
Then $|\tilde y - x_i| \leq |\tilde y - y_i| + |y_i - x_i| < 4 \veps^{3/2} + \veps < 2 \veps$ if $\veps$ is small enough.
Thus $\tilde y - P(\tilde y)| < 2 \veps$. 
Note also that $y \in A_i$ and $y_i \in A_i$ (from \eqref{AlongNormalDirectionVoronoi}). Since the set $A_i$ is convex, we conclude that $\tilde{y} \in A_i$. To get to a contradiction we want to show that $|\tilde{y}-x_k | < |\tilde{y}-x_i|$ for some $k$;  this would imply that $\tilde{y} \not \in V_i$ which indeed would be a contradiction given that $\tilde{y} \in A_i$.

Note that $P(\tilde{y}) \in B(x_k,\veps^2)$ for some $k$. Thus
\begin{align}
\begin{split}
 |\tilde{y}-x_k|^2 &\leq  \left( |\tilde{y}- P(\tilde{y})| + |P(\tilde{y}) - x_k|  \right)^2
 \\ & = |\tilde{y}- P(\tilde{y})|^2 + 2 |\tilde{y}- P(\tilde{y})|\cdot |P(\tilde{y}) - x_k| +|P(\tilde{y}) - x_k| ^2
 \\ & \leq |\tilde{y}- P(\tilde{y})|^2 + 4 \veps^3 + \veps^4.
 \end{split}
 \label{AuxAppendix0}
\end{align}
Furthermore, note that 
\begin{align}
\begin{split}
|\tilde{y}-x_i|^2 &= |y_i-x_i|^2 + |\tilde{y}- y_i|^2
\\& = g(y_i)^2+  |\tilde{y}- y_i|^2
\\& = g(\tilde{y})^2 + g(y_i)^2 - g(\tilde{y})^2 + |\tilde{y}- y_i|^2
\\& \geq   |\tilde{y}-P(\tilde{y})|^2 - | g(y_i)^2-  g(\tilde{y})^2|  + |\tilde{y}- y_i|^2.
\end{split}
\label{AuxAppendix1}
\end{align}
Since $g$ is smooth in $\overline{\partial_{\veps_0} D}$, there exists $M$ such that $M \geq  \| D^2 g(x) \|$ for all $x \in \overline{\partial_{\veps_0} D}$. 
By \eqref{GradientSignedDistance}, the gradient of the signed distance function $g$ at the point $y_i$ is equal to  $ \vec{n}_i $. Since $\tilde{y} - y_i $ is orthogonal to $\vec{n}_i$,
by Taylor expansion
$ |g(\tilde{y}) - g(y_i)| = |g(\tilde{y}) - g(y_i) - Dg(y_i) \cdot (\tilde y - y_i)| \leq M |\tilde{y}-y_i|^2 $.
Thus $|g(\tilde{y})^2 - g(y_i)^2| = |g(\tilde{y}) - g(y_i)|\cdot |g(\tilde{y}) + g(y_i)| \leq 3M \veps |\tilde{y}-y_i|^2 $. Using \eqref{AuxAppendix1} we deduce that
$$ |\tilde{y}-x_i|^2 \geq  |\tilde{y}-P(\tilde{y})|^2  + ( 1-3M\veps)  |\tilde{y}- y_i|^2,  $$
Therefore for small enough $\veps>0$ 
$$ |\tilde{y}-x_i|^2 \geq  |\tilde{y}-P(\tilde{y})|^2  +5 \veps^{3}.  $$
Combining the previous inequality with \eqref{AuxAppendix0} we deduce that $|\tilde{y} - x_i| > |\tilde{y}-x_k|$. This proves the claim.

\medskip

Consider the circular cylinder whose axis is the line passing through the point $x_i$ with direction $\vec{n}_i$ and whose radius is $4 \veps^{3/2}$.   We let $C_i^+$ be the portion of the cylinder contained in $S_i$. 

By \eqref{AlongNormalDirectionVoronoi} we can find a circular cylinder of smaller radius, whose axis is the same as that of $C_i^+$, but such that the portion of it contained in $S_i$, denoted by $C_i^{-}$, satisfies:
$$  C_i^{-} \subseteq A_i  \subseteq C_i^{+}.$$
See Figure \ref{sunset}.

\begin{figure}[h] 
\centering
\includegraphics[width=8cm]{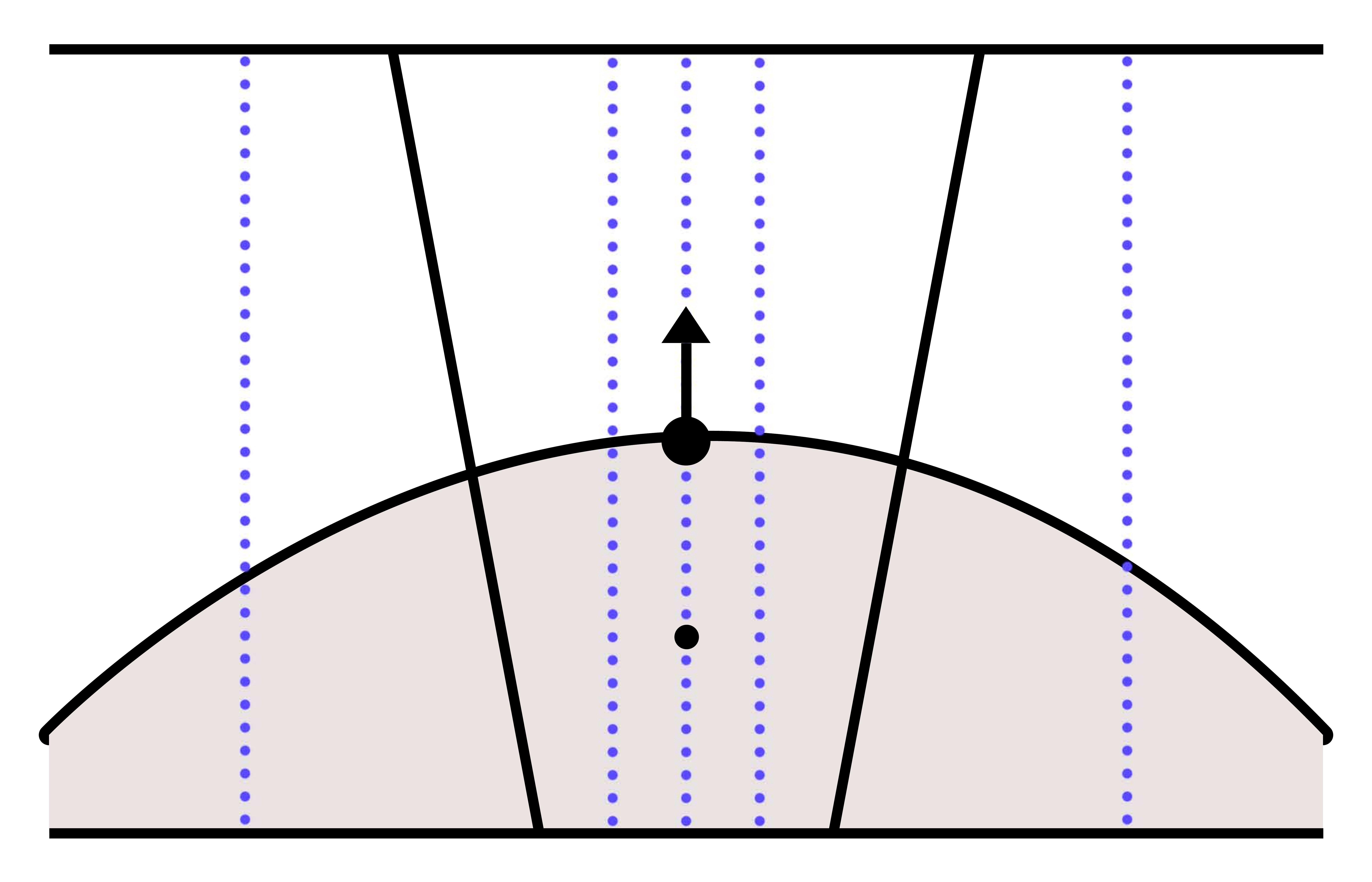}
\put(-70,20){\Large $D$}
\put(-60,128){$C_i^+$}
\put(-113,128){$C_i^-$}
\put(-220,14){$T_i^-$}
\put(-220,128){$T_i^+$}
\put(-220,70){$S_i$}
\put(-112,63){$x_i$}
\put(-112,80){$\vec{n}_i$}
\put(-112,30){$\tilde{x}_i$}
\put(-150,90){$A_i$}
\caption{  }
\label{sunset}
\end{figure}
\emph{Claim 2.} Let $0<\veps< \frac{\veps_0}{2}$ be small enough. Then, there exists a map $\Phi_i: A_i \cap \overline{D} \rightarrow A_i$ which is a bi-Lipschitz homeomorphism. In particular, since $A_i$ is a closed convex body with nonempty interior, we conclude that $A_i \cap \overline{D}$ is bi-Lipschitz homeomorphic to the unit cube.  

To prove the claim fix $0< \veps $ so that in particular the conclusions from Claim 1 hold.  From the bound on the second derivative of $g$ and since the radius of $C_i^+$ is $4\veps^{3/2}$, we deduce that there exists a universal constant $L>0$  such that
\begin{equation}
| \vec{n}_z - \vec{n}_i| \leq  L \veps^{3/2}, \: \: \forall z \in \partial D \cap A_i,
\label{DiffrenceNormal}
\end{equation}
due to the fact that $A_i \subseteq C_i^+$.

We now turn to constructing the bi-Lipschitz mapping between $\overline D \cap A_i$ and $A_i$.
We do that by linear mappings along rays emanating from $\tilde x_i$. 
Consider $\mathcal{S}^{d-1}$ the set of all unit vectors in $\R^d$.
For $\vec{n} \in \mathcal{S}^{d-1}$ define $s_{\vec{n}}$ and $t_{\vec{n}}$ by 
$$  s_{\vec{n}}:= \sup \left\{ s>0  \: : \:  \tilde{x}_i + s \vec{n}  \in \overline{D} \cap A_i \right\} ,  $$
$$ t_{\vec{n}}:= \sup \left\{ t>0  \: : \:  \tilde{x}_i + t\vec{n}  \in  A_i  \right\} .      $$

Since $C_i^- \subseteq A_i\subseteq C_i^+$, we deduce that both functions $\vec{n} \in \mathcal{S}^{d-1} \mapsto s_{\vec{n}}$ and  $\vec{n } \in \mathcal{S}^{d-1} \mapsto t_{\vec{n}}$ are bounded above and below by positive constants. 

Now, note that for every $\vec{n} \in \mathcal{S}^{d-1}$, we have $s_{\vec{n}} \leq t_{\vec{n}}$. Moreover, by \eqref{DiffrenceNormal} and the fact that $A_i \subseteq C_i^+$, we deduce that if $ s_{\vec{n}} < t_{\vec{n}} $ then 
$$ | \vec{n}_i - \vec{n}| \leq L \veps^{3/2},   $$
where $L$ is a universal constant which is not necessarily the same as in \eqref{DiffrenceNormal}. In particular, by choosing $\veps$ to be small enough we can assume that if $s_{\vec{n}} < t_{\vec{n}}$ then, the ray starting at $\tilde{x}_i$ with direction $\vec{n}$ only intersects $\partial D \cap A_i$ at one point. This fact, together with the smoothness of the  outer normal vector field implies that the map  $\vec{n} \in \mathcal{S}^{d-1} \mapsto s_{\vec{n}}$ is Lipschitz. On the other hand, since the set $A_i$ is a convex set with piecewise smooth boundary ( a convex polytope), we deduce that the function $\vec{n } \in \mathcal{S}^{d-1} \mapsto t_{\vec{n}}$ is Lipschitz as well.

Consider the map $\Phi_i: \overline{D}\cap A_i \rightarrow A_i$ defined as follows. Set $\Phi_i(\tilde{x}_i)= \tilde{x}_i$.  For  $x \in \overline{D} \cap A_i $, $x \not = \tilde{x}_i$ we can write  $x = \tilde{x}_i + s \vec{n} $, for some $\vec{n} \in \mathcal{S}^{d-1}$ and for some $0 < s \leq s_{\vec{n}}$; we let $\Phi_i(x)$ be 
\begin{equation*}
 \Phi_i(x):= \tilde{x}_i + \frac{st_{\vec{n}}}{s_{\vec{n}}} \vec{n}.
\end{equation*}

Since both functions $\vec{n} \in \mathcal{S}^{d-1} \mapsto s_{\vec{n}}$ and  $\vec{n } \in \mathcal{S}^{d-1} \mapsto t_{\vec{n}}$ are bounded above and below by positive constants and are Lipschitz, we deduce that the map $\Phi_i$ is a bi-Lipschitz homeomorphism between $\overline{D} \cap A_i$ and $A_i$. This proves the claim.

\emph{Claim 3.} For  any $\veps<1$ it holds that $\partial D \cap ( V_i \setminus S_i )= \emptyset.$ 
To prove this claim, assume for the sake of contradiction that there exists $x \in \partial D \cap ( V_i \setminus S_i ) $. Since $x \not \in S_i$, it follows that $|x-x_i| \geq \veps$. On the other hand, given that $x\in \partial D$, we know there exists $k$ such that $x \in B(x_k, \veps^2)$. Since $\veps <1$, we deduce that $|x-x_k| < |x-x_i|$ and thus $x \not \in V_i$. This is a contradiction.

Now we have all the ingredients needed to prove Proposition \ref{HPropertySmoothDomains}. Indeed, take $\veps>0$ small enough so that  all of the conclusions of all the previous claims hold. From Claim 3, we deduce that every $V_i$ can be partitioned into three convex polytopes. One which intersects $\partial D$, namely $A_i= V_i \cap S_i$ and other two polytopes, one which is contained in $\inte(D^c)$ and another one contained in $D$. We denote the later one by $\hat{A}_i$.  We consider the family $\left\{ A_1, \hat{A}_1, \dots, A_N , \hat{A}_N \right\}$ of convex polytopes. This family covers $D$ and is such that properties (1) and (2) from Definition \ref{DefinitionHProperty} are satisfied. Moreover, given that $\hat{A}_i \subseteq D$ and given that $\hat{A}_i$ is convex, we deduce that $\hat{A}_i$ satisfies property (3) automatically, since all closed convex bodies with piecewise smooth boundary are bi-Lipschitz homeomorphic. Finally, Claim 2 implies that property (3) holds for each of the $A_i$. All together this implies that $D$ satisfies the (WP) property.

\bigskip
\noindent {\bf Acknowledgments.}
 DS is grateful to  NSF (grant DMS-1211760) for its support. The authors are thankful to Zilin Jiang and to Felix Otto for enlightening conversations. The authors would like to thank the Center for Nonlinear Analysis of the Carnegie Mellon University for its support.

\bibliography{Biblio}
\bibliographystyle{siam}

\end{document}